\documentclass[11pt,reqno,tbtags]{amsart}

\usepackage{amsthm,amssymb,amsfonts,amsmath}
\usepackage{amscd} 
\usepackage[foot]{amsaddr}
\usepackage{mathrsfs}
\usepackage{subfigure, graphicx}
\usepackage{vmargin}
\usepackage{setspace}
\usepackage{paralist}
\usepackage{stmaryrd} 

\usepackage{hyperref}
\usepackage[numbers, sort&compress]{natbib}

\theoremstyle{plain}
\newtheorem{theorem}{Theorem}
\newtheorem{lemma}[theorem]{Lemma}
\newtheorem{proposition}[theorem]{Proposition}
\newtheorem{corollary}[theorem]{Corollary}

\newtheorem{fact}[theorem]{Fact}
\theoremstyle{definition}

\newtheorem{definition}[theorem]{Definition}

\theoremstyle{definition}

\newcommand{\rst}[1]{\ensuremath{{\mathbin\upharpoonright}%
\raise-.5ex\hbox{$#1$}}} 
\newcommand{\Bin}{\mathbf{Bin}}
\newcommand{\Po}{\mathbf{Po}}
\newcommand{\Ber}{\mathbf{Ber}}
\newcommand{\eps}{\varepsilon}
\newcommand{\E}{\mathbb{E}}
\renewcommand{\Pr}{\mathbb{P}}

\providecommand{\cH}{}
\renewcommand{\cH}{\mathcal{H}}

\newcommand{\diam}{\text{diam}}
\newcommand{\dist}{\text{dist}}

\newcommand{\ESL}{E_{\text{SL}}}
\newcommand{\sstar}{\Sigma^\star}

\newcommand{\cT}{\mathcal{T}}

\providecommand{\R}{}

\providecommand{\N}{}

\renewcommand{\R}{\mathbb{R}}

\renewcommand{\N}{{\mathbb N}}

 \newcommand{\bag}{\begin{align}}
\newcommand{\bags}{\begin{align*}}
\newcommand{\eag}{\end{align*}}
\newcommand{\eags}{\end{align*}}

\newcommand{\pran}[1]{\left(#1\right)}

\providecommand{\eps}{}
\renewcommand{\eps}{\epsilon}
\providecommand{\ora}[1]{}
\renewcommand{\ora}[1]{\overrightarrow{#1}}

\newcommand{\eqdist}{\ensuremath{\stackrel{\mathrm{d}}{=}}}

\hypersetup{
    bookmarks=true,         
    unicode=false,          
    pdftoolbar=true,        
    pdfmenubar=true,        
    pdffitwindow=true,      
    pdftitle={My title},    
    pdfauthor={Author},     
    pdfsubject={Subject},   
    pdfnewwindow=true,      
    pdfkeywords={keywords}, 
    colorlinks=true,       
    linkcolor=blue,          
    citecolor=blue,        
    filecolor=blue,      
    urlcolor=blue           
}

\numberwithin{equation}{section}
\numberwithin{theorem}{section}

\newcommand{\psub}[2]{{\mathbf P}_{#1}\left(#2\right)}

\newcommand{\Esub}[2]{{\mathbf E_{#1}}\left(#2\right)}
\newcommand{\esub}[1]{{\mathbf E_{#1}}}

\usepackage[textsize=footnotesize,color=green!40]{todonotes}

\begin{document}

\title{Diameter and Stationary Distribution of Random $r$-out Digraphs}
\author{Louigi Addario-Berry$^*$}
\address{$^*$Department of Mathematics and Statistics, McGill University, Montreal, Canada}
\author{Borja Balle$^\dagger$ \and Guillem Perarnau$^\dagger$}
\address{$^\dagger$School of Computer Science, McGill University, Montreal, Canada}
\email{louigi.addario@mcgill.ca}
\email{bballe@cs.mcgill.ca}
\email{p.melliug@gmail.com}

\date{\today}

\begin{abstract}
Let $D(n,r)$ be a random $r$-out regular directed multigraph on the set of vertices $\{1,\ldots,n\}$. In this work,
we establish that for every $r \ge 2$, there exists $\eta_r>0$ such that $\diam(D(n,r))=(1+\eta_r+o(1))\log_r{n}$.
Our techniques also allow us to bound some extremal quantities related to the stationary distribution of a simple random walk on $D(n,r)$.
In particular, we determine the asymptotic behaviour of $\pi_{\max}$ and $\pi_{\min}$, the maximum and the minimum values of the stationary distribution.
We show that with high probability $\pi_{\max} = n^{-1+o(1)}$ and $\pi_{\min}=n^{-(1+\eta_r)+o(1)}$. Our proof shows that the vertices with $\pi(v)$ near to $\pi_{\min}$ lie
at the top of ``narrow, slippery towers''; such vertices are also responsible for increasing the diameter from $(1+o(1))\log_r n$ to $(1+\eta_r+o(1))\log_r{n}$.
\end{abstract}

\maketitle


\section{Introduction}\label{sec:intro}

Call a random directed graph $D$ with vertices $V(D)=\{v_1,\ldots,v_n\}$ a {\em random $r$-out digraph}
if each vertex in $V(D)$ has out-degree $r$, and the $nr$ heads of edges in $E(D)$ are iid and uniformly distributed over $V(D)$.
We allow digraphs to have multiple edges and loops.
It is useful to have a canonical construction: for each pair $(i,j) \in [n]\times [r]$, let $L_{i,j}$ be a uniformly random element of $[n]$, and write $D(n,d)$ for the random $r$-out digraph with vertex set $[n]=\{1,\ldots,n\}$ and edge set $\{(i,L_{i,j}): (i,j) \in [n] \times [r]\}$.

Given a digraph $D$, for $u,v \in V(D)$ we write $\dist(u,v)=\dist_D(u,v)$ for the number of edges in a shortest oriented path from $u$ to $v$, or set $\dist(u,v)=\infty$ if there exists no such path. The \emph{diameter} of $D$ is
\begin{equation}\label{eq:diam_def}
\diam(D) = \max\{\dist(u,v): u,v \in [n], \dist(u,v) < \infty\}.
\end{equation}

Say $D$ is {\em strongly connected} if $\dist(u,v) < \infty$ for all $u,v \in V(D)$.
An induced subgraph $D[S]$ of $D$ is a {\em strongly connected component} of $D$ if $D[S]$ is strongly connected but for all $S'$ with $S \subset S'$, $D[S']$ is not strongly connected.
Given $S \subset V(D)$, say that $D[S]$ is {\em attractive} if for all $v \in V(D)$ there is a directed path from $v$ to $S$. It is easily seen that a digraph can contain at most one attractive strongly connected component $D[S]$.

If $D$ is strongly connected then a simple random walk on $D$ has a unique stationary distribution $\pi=\pi_D$; in this case we write $\pi_{\max}(D)=\max\{\pi_D(v):\; v\in V(D)\}$ and $\pi_{\min}(D)=\min\{\pi_D(v):\; v\in V(D)\}$, respectively.
The diameter, and the values $\pi_{\max}$ and $\pi_{\min}$, are natural extremal parameters associated with a digraph.
In order to study cases where $D$ is not necessarily strongly connected, we write $D_0=D_0(n,r)$ for the strongly connected component of $D(n,r)$ with the largest number of vertices (if there is more than one such component, $D_0$ is the one whose smallest labelled vertex is minimal).

Let $\lambda_r=\max\{\lambda:1-\lambda=e^{-r\lambda}\}$, and let
\begin{equation}
\eta_r = \frac{1}{\log_r(1-\lambda_r)^{-1}-1} = \frac{\log r}{\lambda_r r - \log r} \enspace.
\end{equation}
Observe that $\lambda_r\to 1$ and $\eta_r\to 0$ when $r\to \infty$.
A sequence of random variables $X_n$ converges to $X$ {\em in probability} if
for every $\eps > 0$, $\Pr(|X_n-X| > \eps) \to 0$ as $n \to \infty$. If $X_n/Y_n
\to X$ in probability then we also write $X_n = (X+o_p(1))Y_n$.

The study of $D(n,r)$ was initiated by Grusho, who showed that the size of its largest strongly connected component satisfies $|V(D_0(n,r))|=(1+o_p(1))\lambda_r \cdot n$~\cite{grusho1973limit}~(we remark that its size is also the asymptotic size of the giant component in the Erd{\H o}s-R\'enyi random graph $G(n,r/n)$). Because $\lambda_r > 1/2$ for all $r \ge 2$, it follows that $D_0(n,r)$ is with high probability\footnote{Here and for the remainder of the paper, {\em with high probability}, or whp, means with probability tending to $1$ as $n \to \infty$.} the unique strongly connected component of its size. Motivated by the average-case analysis of algorithms for the minimization of Deterministic Finite Automata (DFA), Grusho's result has been recently rediscovered by different sets of authors~\cite{aryeh,cai2015graph,carayol2012distribution}.

The diameter of $D(n,r)$ was first studied by Trakhtenbrot and Barzdin in~\cite[Theorem 5.5]{trakbarz}, who showed that for every $r\geq 2$ there exists a constant $C_r\geq 1$ such that with high probability $\diam(D)\leq C_r \log_r{n}$. Since $D(n,r)$ is $r$-out regular we always have the trivial lower bound $\diam(D)\geq \lceil\log_r{(n-1)}\rceil$.

In \cite{grusho1973limit}, Grusho also showed that the unique largest strongly connected component $D_0$ is attractive with high probability. More recently, Balle~\cite{ballearxiv} showed that $D_0(n,r)$ is whp aperiodic, and so $D_0$ is ergodic. It follows that whp, the law of the position of a particle performing a simple random walk on $D(n,r)$ converges to $\pi_{D_0(n,r)}$.

The contribution of this paper is to determine the first order asymptotic behaviour of $\diam(D(n,r))$, $\pi_{\max}(D_0(n,r))$ and $\pi_{\min}(D_0(n,r))$, as $n$ becomes large.
\begin{theorem}\label{thm:main}
For every $r \ge 2$,  we have $\diam(D(n,r)) = (1 + \eta_r + o_p(1)) \log_r n$ and
$\diam(D_0(n,r)) = (1 + \eta_r + o_p(1)) \log_r n$.

\end{theorem}
\begin{theorem}\label{thm:dist}
For every $r \ge 2$, we have $\pi_{\max}(D_0(n,r)) = n^{-1 + o_p(1)}$ and
$\pi_{\min}(D_0(n,r)) = n^{-(1 + \eta_r) + o_p(1)}$.
\end{theorem}

\noindent{\bf Remarks.} 
\begin{itemize}
\item The results obtained can be easily transferred to random \emph{simple} $r$-out digraphs. Let $D_{\mathrm{sim}}(n,r)$ be chosen uniformly at random from the set of directed simple  graphs (no loops or multiple edges) with vertex set $[n]$ such that each vertex has out-degree $r$. The conditional distribution of $D(n,r)$, given that it is simple, is precisely that of $D_{\mathrm{sim}}(n,r)$.
Furthermore, it is not hard to show (see~\cite{mckay1984asymptotics}) that
\begin{equation*}
\Pr(D(n,r)\text{ is simple})= e^{-\Theta(r^2)}\;.
\end{equation*}
In particular, this probability is bounded away from zero for fixed $r$, so any
property that holds whp for $D(n,r)$ also holds whp for
$D_{\mathrm{sim}}(n,r)$.

\item It is not hard to deduce from our arguments that for all $u,v \in V(D(n,r))$, conditional on the event that $v \in D_0(n,r)$, we have $\dist_{D(n,r)}(u,v)=(1+o_p(1))\log_r n$. This shows that the \emph{typical distance} in $D(n,r)$ is $(1+o_p(1))\log_r n$. The argument, in brief, is as follows. First, the lower bound is easy by symmetry since for all $u \in V(D(n,r))$ we have $|N^+_{\le d}(u)| \le \sum_{i=0}^d r^i < r^{d+1}$. For the lower bound, Proposition~\ref{prop:direct} tells us that if $N_k^-(v) \ge \log^4 n$ then with high probability $\dist(u,v) \le k+\log_r n$. By Lemma~\ref{lem:riordan} and Proposition~\ref{lem:towers}, with $k = (\log \log n)^2$, it follows straightforwardly that $\Pr(N_k^-(v) <\log^4 n\mid v \in D_0(n,r)) = o(1)$, and the result follows. We leave the details to the interested reader. 

\item The random $r$-out, $s$-in digraph $D(n,r,s)$ is defined similarly to $D(n,r)$, but each vertex chooses $s$ in-neighbours as well as $r$ out-neighbours, all independently and uniformly at random; see \cite{fennerfrieze}. In particular, $D(n,r) \eqdist D(n,r,0)$\footnote{For any two random variables $X$ and $Y$, we use the notation $X\eqdist Y$ to denote that the corresponding probability distributions are equal.}. It may be interesting to consider the diameter and the stationary distribution for $D(n,r,s)$ when $s \ne 0$. One case follows from Theorem~\ref{thm:main}: since the diameter of a digraph is the same as the diameter of the digraph obtained by flipping the direction of all the edges, $\diam(D(n,0,r)) \eqdist \diam(D(n,r,0))$. In contrast, studying the stationary distribution of $D(n,0,r)$ seems less interesting: typically there will be many vertices with no out-edges where a simple random walk will eventually become stuck.
\end{itemize}

\textbf{Outline.} The paper is organized as follows. We start in Section~\ref{sec:related} by discussing our motivation for addressing these problems and by putting our results in the context of other models for random (di)graphs. In Section~\ref{sec:preliminaries} we introduce the notation that will be used throughout the paper and state some basic concentration inequalities and facts about branching processes. In Section~\ref{sec:upper_proof} we finish the proof of the upper bound on the diameter of $D(n,r)$ (Theorem~\ref{thm:main}) assuming some technical estimates. The breadth-first search procedure that will be used to explore the graph is described in Section~\ref{sec:bfs_and_cond}. In Section~\ref{sec:lemmas} we study the behaviour of the in-neighbourhoods of $D(n,r)$ by comparing them with Poisson Galton-Watson trees, while in Section~\ref{sec:out_tech} we study its out-neighbourhoods. In Section~\ref{sec:upper}, we prove the technical estimates, completing the proof of the upper bound given in Section~\ref{sec:upper_proof}. The proof of the lower bound on the diameter of $D(n,r)$ (Theorem~\ref{thm:main}) occupies Section~\ref{sec:lower}. We conclude the paper by proving Theorem~\ref{thm:dist} in Section~\ref{sec:stat}.


\section{Motivation and Related Work}\label{sec:related}

One of our main motivations for the study of $D(n,r)$ comes from the analysis of random deterministic finite automata (DFAs).
In this section we describe the particular problem that leads us to study
the diameter and stationary distribution of these objects.

\subsection{Learning Random Deterministic Finite Automata}

A \emph{deterministic finite automaton} (DFA) over an alphabet $\Sigma=\{\sigma_1,\ldots,\sigma_r\}$ is given by a set $V=\{v_1,\ldots,v_n\}$ and a function $L:[n] \times [r] \to [n]$. We think of the pair $(V,L)$ as specifying a
directed multigraph $D$ with vertices $V$ and edges $\{(v_i,L(i,j)): i \in [n], j \in [r]\}$; every vertex of $D$ has out-degree $r$, and the $r$ edges leaving a vertex $v$ are labeled with distinct symbols from $\Sigma$.
In addition, a DFA is equipped with a distinguished vertex $s$ called the
\emph{initial} state, and with a binary labelling $B:V(D) \to \{0,1\}$; the vertices in $B^{-1}(\{1\})$ are the
{\em accepting} states of the DFA. The DFA is formally given by the tuple $Q=(V,\Sigma,L,s,B)$.


Let $\sstar$ denote the set of all finite strings with symbols in $\Sigma$.
Words $w = w_1w_2 \dots w_t \in \sstar$ correspond to walks $x_0(w),x_1(w),\dots,x_t(w)$ on $V$:
$x_0=s$ and, for $1 \le i \le t$, $x_i$ is reached from $x_{i-1}$ by following the edge with label $w_i$.
We write $Q(w)=x_t(w)$ for the final state of the walk. The DFA {\em accepts} the word $w$ if $B(Q(w))=1$.
The set $L(Q) = \{w \in \sstar: B(Q(w))=1\}$ is the {\em language} recognized by the DFA.
The set of languages recognized by some DFA are precisely the \emph{regular languages}.

To see the connection with random out-regular graphs, observe that
we may build a uniformly random DFA with $n$ labelled states and alphabet of size $r$ as follows. Let
$D(n,r)$ be as in the first paragraph of the paper, using the
random variables $(L_{i,j}: (i,j) \in [n]\times [r])$. Then for $(i,j) \in [n]\times[r]$, let $L(i,j)=L_{i,j})$;
equivalently, assign label $\sigma_i$ to edge $(i,L_{i,j})$.
Choose the starting state $s$ uniformly at random from $[n]$, and choose
$B$ uniformly at random from the set of functions $f:[n] \to \{0,1\}$.

DFAs and regular languages play a crucial role in language theory and there is a
vast literature on algorithms over DFAs, ranging from minimization and equivalence
testing, to synthesis, learning and composition.

Learning regular languages from different sources of information is a prominent
problem in computational learning theory \cite{kearnsbook}, which is most often
studied within the context of so-called grammatical inference problems
\cite{colinbook}.
A prominent problem in this area concerns the possibility of learning regular
languages under the \emph{probably approximately correct (PAC)} learning model
introduced by Valiant \cite{valiantpac}.
Roughly speaking, this asks for an efficient
algorithm such that, when supplied with a large enough sample containing iid
strings drawn from some arbitrary probability distribution $\mu$ on $\sstar$ and
labels indicating whether each string belongs to some hidden regular language,
the algorithm outputs a representation of a regular language (e.g.\ a DFA) which
is close to the hidden regular language in a sense that depends on the
distribution which generated the sample strings.
Several results from the 90's indicate that, in its full generality, PAC
learning of DFAs is hard due to complexity-theoretic as well as cryptographic reasons
\cite{kearnsvaliant,pittwarmuth} (see also the recent strengthened result
\cite{chalermsook2014pre}).
A natural question to ask in such a scenario is whether there exists a
reasonable simplification of the problem for which a positive answer is
possible.
This requires one to come up with scenarios that rule out the worst-case
problems arising from specially crafted regular languages and distributions
over examples appearing in the proofs of the aforementioned lower bounds.

%

One possibility is to study the average case. This approach can be formalized by
considering regular languages defined by random DFAs.
In particular, one can ask for an algorithm that with high probability (as the
number of the states in the DFA goes to infinity) can learn the regular language
recognized by a random DFA.
There exists evidence suggesting that such relaxation might not be
enough to achieve efficient learning in general: it was recently showed by
Angluin et al.\ that generic instances of DFA (as well as decision trees and DNF
formulas) are hard to learn from statistical queries when examples can be
sampled from an arbitrary distribution \cite{angluinlower}.
Nevertheless, prior to Angluin et al.'s result it was showed that generic
decision trees and generic DNF formulas can be efficiently learned when samples
are drawn according to the uniform distribution~\cite{randomdt,randomdnf}.

In view of the panorama described in the previous paragraphs, a natural question
to ask is whether random DFAs can be efficiently learned when sample strings are
drawn from the uniform distribution.
More precisely, one would like to answer the following sorts of questions. Fix a uniformly random DFA $Q$ with states $[n]$ and alphabet $[r]$. Then fix $m \in \N$ and let $(\mathbf{x}_i,i \ge 1)$ be iid words sampled uniformly at random from $[r]^m$.
\begin{enumerate}
\item Given the sequences $(\mathbf{x}_i,i \ge 1)$ and $(B(Q(\mathbf{x}_i)),i \ge 1)$, is it possible to construct a DFA
$\hat{Q}$ that recognizes the same language as $Q$ with high probability?
\item Given the sequences $(\mathbf{x}_i,i \ge 1)$, $(Q(\mathbf{x}_i)),i \ge 1)$ and $(B(Q(\mathbf{x}_i)),i \ge 1)$, is it possible to construct a DFA $\hat{Q}$ that recognizes the same language as $Q$ with high probability?
\end{enumerate}
In both cases, if the answer is yes then it is natural to ask for {\em efficient} algorithms (average case running time polynomial in $n$, $m$, $r$, and any other parameters involved). The questions can be weakened by only requiring that $\hat{Q}$ recognizes the same set of words {\em of length m}. A further weakening is to only require that $\Pr(\hat{Q}(\mathbf{y})=Q(\mathbf{y})) > 1-\eps$ when $\mathbf{y}$ is uniformly distributed over $[r]^m$.

The results in~\cite{angluincomm} establish that in order to answer the second question, it would be sufficient to understand several specific properties of a random walk on a randomly generated DFA.
When a string is sampled from the uniform distribution over $[r]^m$ and is labeled according to the state that it reaches, the label immediately corresponds to the final state of a simple random walk of length $m$ over the DFA starting from the initial state.
Thus, the analysis of the algorithm in \cite{angluincomm} relies on bounds on the diameter, stationary distribution, and mixing time on random $r$-out regular digraphs.
Similar ideas are what led us to the study of the problems discussed in the present paper.

Several other properties of random DFAs have been studied, both in learning theory and in other contexts, using the $D(n,r)$ model.
For example, first Korshunov's group, and later Nicaud's group, have studied the probability that random DFA exhibit particular structures, mainly motivated by the analysis of sample and reject algorithms for enumeration of subclasses
of automata (see \cite{DBLP:conf/mfcs/Nicaud14} and references therein).
Motivated by worst-case hardness results for learning a DFA, Angluin and co-authors have used properties of random DFAs to study the problem of learning a generic DFA \cite{angluin2009learning,angluinlower}.
The average-case complexity of DFA minimization algorithms has also received some attention recently \cite{bassino2012average,de2013brzozowski}.
Finally, a series of results have led to a solution of the long-standing \v{C}ern{\'y} conjecture about synchronization of finite automata in the case of random DFA~\cite{DBLP:journals/corr/abs-1304-5774,DBLP:journals/corr/Nicaud14,skvortsov2010synchronizing}.

\subsection{Diameter and stationary distribution of other random graph models}\label{sub:diam_models}

In this subsection we describe some previous results on the diameter and the stationary distribution of certain random graph models and relate them to Theorem~\ref{thm:main} and to Theorem~\ref{thm:dist}. This provides an intuition for the results we have obtained on the diameter of $D(n,r)$. We consider the following models of random (di)graphs.
\begin{itemize}
\item For $p \in [0,1)$, $G(n,p)$ is the random graph with vertex set $[n]$ in which every edge is included independently with probability $p$.
\item For $d \in \N$, $G(n,d)$ is the random $d$-regular simple graph with vertex set $[n]$ chosen uniformly at random among all such graphs.
\item  For $p \in [0,1)$, $D(n,p)$ is the random digraph with vertex set $[n]$ in which  every oriented edge is included independently with probability $p$.
\end{itemize}
For an undirected graph $G=(V,E)$ and $u,v \in V$ we write $\dist_G(u,v)$ for the minimum number of edges in a path from $u$ to $v$,
or set $\dist_G(u,v)=\infty$ if there exists no such path. The diameter of $G$ is then defined just as in~\eqref{eq:diam_def}. Bollob\'as and Fernandez de la Vega~\cite{bollobas1982diameter} studied the diameter of $G(n,d)$ and showed that for every integer $r \ge 2$, we have
\begin{align}\label{eq:diam_RRG}
\diam(G(n,r+1))= (1+o_p(1))\log_r{n}\;.
\end{align}
The diameter of $G(n,p)$ was recently studied by Riordan and Wormald~\cite{riordan}, who showed that for every constant $r>0$, we have
\begin{align}\label{eq:diam_RG}
\diam(G(n,(r+1)/n))=(1+2\eta_r+o_p(1))\log_r{n}\;.
\end{align}
In fact, they proved a stronger result, showing convergence in distribution of the diameter after appropriate recentering and rescaling. The extra term $2\eta_r$ is essentially due to the existence of ``remote'' vertices in the giant component of $G(n,(r+1)/n)$, whose neighbourhoods are exceptionally small up to distance about $\eta_r \log_r n$.

Our result on the diameter of $D(n,r)$ from Theorem~\ref{thm:main} can be related to~\eqref{eq:diam_RRG} and~\eqref{eq:diam_RG} in the following way. Given $u,v \in [n]$, one way to determine $\dist_{D(n,r)}(u,v)$ is to perform an {\em outward}  breadth-first search (BFS) starting at $u$, to perform an {\em inward} BFS (i.e.\ following edges from head to tail) starting at $v$, and to stop at the first time the two searches uncover a common vertex. (See Section~\ref{sub:bfs} for a careful definition of breadth-first search.) This technique was used by Bollob\'as and Fernandez de la Vega in~\cite{bollobas1982diameter}. Since the BFS explores vertices in order of distance, such a procedure is guaranteed to build a shortest path from $u$ to $v$.

On the one hand, in the outward BFS of $D(n,r)$ starting from $u$, every vertex has exactly $r$ out-edges when explored. Similarly, in a BFS exploration of $G(n,r+1)$, when a vertex $v$ is discovered via an edge from one of its neighbours, this leaves $r$ edges to unveil when $v$ is itself explored (unless $v$ is discovered multiple times, which at least at the start of the BFS is unlikely). Thus, a BFS of $G(n,r+1)$ looks similar to an outward BFS of $D(n,r)$.

On the other hand, in the inward BFS of $D(n,r)$ starting from $v$ (or at least near the start of the process) the number of in-edges arriving at a vertex are roughly distributed as a Binomial random variable with $n$ trials and success probability $r/n$. Thus, a BFS of  $G(n,(r+1)/n)$ looks similar to an inward BFS of $D(n,r)$.

The preceding paragraphs suggest that shortest paths in $D(n,r)$ are in some sense hybrids of shortest paths in $G(n,r+1)$ and in $G(n,(r+1)/n)$. This, together with~\eqref{eq:diam_RRG} and~\eqref{eq:diam_RG}, provides some intuition for the value of the diameter of $D(n,r)$ from Theorem~\ref{thm:main}: it is the average of the limit values in those formulae.

There is interesting related work on distances in graphs with random edge weights. We mention in particular the paper of Janson \cite{janson1999one} on typical and extreme distances in randomly edge-weighted complete graphs, and the subsequent work by Bhamidi and van der Hofstad~\cite{bhamidi2013diameter}, which establishes distributional convergence for the diameter.

To conclude this section, we discuss the stationary distribution of a simple random walk in these other models. While in undirected graphs the stationary distribution (if it exists) is completely determined by the degrees of the vertices, this is not the case in directed graphs. Cooper and Frieze~\cite{cooper2012stationary} give a very precise description of the stationary distribution of $D(n,c/n)$ when $c=c(n)>(1+\eps)\log{n}$, for any constant $\eps>0$, and use their result to compute the cover time of $D(n,c/n)$. It is worth noticing that for such values of $c$, both the in-degrees and out-degrees are of logarithmic order and concentrated around their expected values, which turns to be very useful for the analysis. It seems harder to find an interesting question about the stationary distribution of $D(n,c/n)$ when $c=c(n)<(1-\eps)\log{n}$ since like in random $r$-in regular digraphs, typically there are vertices with no out-edges.

\section{Notation and preliminaries}\label{sec:preliminaries}
We write $[n] = \{1,2,\ldots,n\}$, $\N = \{1,2,\ldots\}$, and $\N_0 = \{0,1,2,\ldots\}$. The notation $A \subset B$ allows that $A=B$; we write $A \subsetneq B$ for strict containment. Unless we explicitly indicate otherwise, asymptotic notations will always refer
to the case $n \to \infty$. We omit floors and ceilings when doing so improves readability.
All logarithms are natural unless a subscript specifies otherwise.

For any two random variables $X$ and $Y$, we use the notation $X\eqdist Y$ to denote that the corresponding probability distributions are equal.
For random variables $X,Y$, we write $X \preceq Y$, and say $X$ is stochastically dominated by $Y$, if $\Pr(X \le t) \ge \Pr(Y \le t)$ for all $t \in \R$. We say $X_1,\ldots,X_k$ are {\em independently} stochastically dominated by $Y_1,\ldots,Y_t$ if $\Pr(X_i \le t_i, 1 \le i \le k) \ge \prod_{i=1}^k \Pr(Y_i \le t_i)$, for all $(t_1,\ldots,t_k) \in \R^k$.


\subsection{Digraphs}
Let $D=(V(D),E(D))$ be a directed graph. For $S, S' \subseteq V(D)$ let
\begin{equation*}
E(S,S') = E_D(S,S') = \left\{ (u,v) \in E \middle| u \in S, v \in S' \right\}\,.
\end{equation*}
Given $S \subset [n]$, $D[S]=(S,E(S,S))$ is {\em the subgraph of $D$ induced by $S$}.

%



Given $u \in [n]$ and an integer $k\geq 0$,
write $N_k^+(D,u) = \{ v \in [n] : \dist(u,v) = k \}$ and $N_{\leq k}^+(D,u) = \cup_{j \leq k} N_j^+(u)$.
Similarly, let $N_k^-(D,u) = \{ v \in [n] :
\dist(v,u) = k \}$ and $N_{\leq k}^-(D,u) = \cup_{j \leq k} N_j^-(u)$.
We write $N^+(D,u) = N_1^+(D,u)$ and $N^-(D,u) = N_1^-(D,u)$.
We also let $d_k^+(D,u) = |N_k^+(D,u)|$, and define $d_{\leq
k}^+(D,u)$, $d_k^-(D,u)$, and $d_{\leq k}^-(D,u)$ correspondingly.
We write $N_k^+(u)=N_k^+(D,u)$, etcetera, when $D$ is clear from context.
%
%
%
\subsection{Concentration Inequalities}

We write $\Bin(N,p)$ to denote a Binomial random variable with $N$ trials and success probability $p$.
We write $\Po(r)$ to denote a Poisson random variables with parameter $r$.
We also write $\Ber(p)$ to denote a Bernoulli random variable with success probability $p$.

We will use the following version of Chernoff's bound for large deviations that can be found in~\cite{janson2011random}.

\begin{lemma}[Chernoff's inequality]\label{lem:chernoff}
For any $t\geq 0$ we have
\begin{equation}\label{eq:cher1}
\Pr(\Bin(N,p)\geq  Np+ t)\leq  e^{-\frac{t^2}{2(Np+t/3)}} \enspace,
\end{equation}
and
\begin{equation}\label{eq:cher1b}
\Pr(\Bin(N,p)\leq  Np- t)\leq  e^{-\frac{t^2}{2Np}} \enspace.
\end{equation}
\end{lemma}
We will also use Chebyshev's inequality: for any random variable $X$ and any $t \ge 0$,
$\Pr(|X-\E(X)|\geq t)\leq \frac{\mathbf{\sigma}^2}{t^2}$,
where $\mathbf{\sigma}^2 =\E(X^2)-\E(X)^2$.

\subsection{Trees and branching processes}\label{sub:trees}
In any rooted tree, we view edges as oriented from child to parent. Fix a rooted tree $T$ with root $v=v(T)$. Then for all $u \in V(T)$, $N^-(T,u)$ is the set of children of $u$. For $u \ne v$, let $p(u)=p_T(u)$ be the parent of $u$ in $T$, so $N^+(T,u)=\{p(u)\}$.
For $k \ge 0$, let $T_{\le k}$ be the subtree of $T$ induced by $N^-_{\le k}(T,v) = \{u \in V(T): \dist_T(u,v)\le k\}$; we view $T_{\le k}$ as rooted at $v$. Also, write $T_k= N^-_k(T,v)=\{u \in V(T): \dist_T(u,v)=k\}$.

A {\em plane tree} is a rooted tree in which the children of each node have a left-to-right order. Given a plane tree $T$, there is a canonical labelling of $V(T)$ by distinct elements of $\{\emptyset\} \cup \bigcup_{i \ge 1}\N^i$, as follows.  The root $v$ has label $\emptyset$; its children are labelled from left to right as $1,\ldots,|N^-(T,v)|$. Given $u \in V(T_k)$ with label $w_1w_2\dots w_k$, the children of $u$
are labelled from left to right as $(w_1\dots w_k i, 1\le i \le |N^-(T,u)|)$.


Conversely, given a rooted tree $T$ with $t$ vertices and an ordering of $V(T)$ as $v_1,\ldots,v_t$, say $w \in V(T)$ has index $j$ if $w=v_j$, for $1\leq j\leq t$. We view $V(T)$ as a plane tree using the convention that the children of each vertex are listed from left to right in increasing order of index. If $V(T) \subset \N$ then we always use the ordering inherited from $\N$. Thus, for a {\em rooted} tree $T$ with $V(T) \subset \N$, and a {\em plane} tree $T'$, we say $T$ and $T'$ are isomorphic, and write $T \cong T'$, if $T$ and $T'$ are identical when viewed as plane trees.

Finally, fix a non-negative, integer-valued random variable $\xi$. A {\em Galton-Watson tree} with branching mechanism $\xi$ is the random, potentially infinite family tree $\cT^\xi$ of a branching process started from a single individual, in which each individual reproduces independently according to $\xi$ (i.e. the number of offspring of each individual has the distribution
of $\xi$). The random tree $\cT^\xi$ is naturally viewed as a plane tree; see \cite{legall06rrt} for details and a careful construction.
If $\xi$ is $\Po(r)$ distributed we call $\cT^{\xi}$ a Poisson$(r)$ Galton-Watson tree.

\section{Theorem~\ref{thm:main}: upper bound}\label{sec:upper_proof}
In this section we describe our proof technique for the upper bound in Theorem~\ref{thm:main}, and prove the theorem assuming two technical estimates. For the remainder of the section let $D=D(n,r)$ and write $d_k^-(v)=d_k^-(D,v)$, $N_k^-(v)=N_k^-(D,v)$ etcetera.

In order to derive an upper bound on the diameter of $D$ we first show that for any fixed vertex $v$, with high probability the in-neighbourhood $N_{k}^-(v)$ is either empty or large for $k$ slightly larger than $\eta_r \log_r n$.

\begin{lemma}\label{lem:inverse}
Fix $v\in [n]$, let $k_0(v)=\min\{k:\;d^-_{k}(v)\not\in (0,\log^4{n})\}$. Then for every $\eps\in (0,1/10)$, there exists $\delta>0$ such that
\[
\Pr\left(k_0(v) > (\eta_r+\eps)\log_r{n} \mbox{ or } d_{\le k_0(v)}^-(v) \ge \log^7 n\right)= O(n^{-(1+\delta)})\, .
\]
\end{lemma}
Next, if $v \in [n]$ has $N^+(v)=\{v\}$ (i.e., if all edges leaving $v$ are self-loops) then call $v$ a {\em loop vertex}.
Let $\ESL$ be the event that $D$ contains some loop vertex.
Each vertex is independently a loop vertex with probability $n^{-r}$, so
\begin{align}\label{eq:prob_ESL}
\Pr(\ESL) = 1-(1-n^{-r})^n= \Theta(n^{1-r})\;.
\end{align}

Note that if $r\geq 3$, the probability of a given vertex being a loop vertex is $O(n^{-3})$. This bound is small enough that it would allow union bounds over pairs of vertices, which would simplify some proofs. Since we aim to prove our result also in the case $r= 2$, we need to be a bit more careful in our computations.
\begin{proposition}\label{prop:direct}
Fix $k,n \in \N$ and $u,v \in [n]$. Let $E_k =\{d^-_k(v)\geq \log^4{n},d^-_{\le k}(v) \le \log^7 n\}$, and
fix a graph $H$ with $v \in V(H) \subset [n]$ such that $\Pr(D[N^-_{\leq k}(v)]=H, E_k, \overline{\ESL})>0$.
Then
\[
\Pr(\dist(u,N^-_{\leq k}(v))> \log_r n - \log_r \log_r n,\overline{\ESL}\mid D[N^-_{\leq k}(v)]=H)= O(n^{-3})\, ,
\]
the preceding bound holding uniformly over $k$ and over all $H$ satisfying the above conditions.
\end{proposition}
We prove Proposition~\ref{prop:direct} in Section~\ref{sec:upper}. In the remainder of the section, we finish the prove of the upper bound from Theorem~\ref{thm:main}, assuming Lemma~\ref{lem:inverse} and Proposition~\ref{prop:direct}.
\begin{proof}[Proof of the upper bound in Theorem~\ref{thm:main}]
Fix $\eps > 0$. We show that $\Pr(\diam (D)\leq (1+\eta_r+\eps)\log_r{n}) = 1-o(1)$. Since $\diam(D_0(n,r)) \le \diam(D(n,r))$, the same bound immediately holds for $D_0(n,r)$.

Let $k^*=(\eta_r+\eps)\log_r{n}$ and let $\ell^* =  \log_r n - \log_r \log_r n$.
 By~\eqref{eq:prob_ESL} for every $r\geq 2$, we have that $\Pr(\ESL)=O(n^{-1})$. So
\begin{align*}
&\Pr(\exists u,v\in [n],\;\dist(u,v) \in (k^*+\ell^*,\infty))\\
\leq &  O(n^{-1})+ \sum_{v\in [n]} \Pr(\exists u\in [n],\;\dist(u,v)\in (k^*+\ell^*,\infty), \overline{\ESL})\, .
\end{align*}
Define $k_0=k_0(v)$ as in Lemma~\ref{lem:inverse}, and let $E = \{k_0\leq k^*, d_{\le k_0}^-(v) < \log^7 n\}$. Then by Lemma~\ref{lem:inverse},
\begin{align*}
& \Pr(\exists u\in [n],\dist(u,v)\in (k^*+\ell^*,\infty), \overline{\ESL})\\
\leq &
\Pr(\overline{E}, \overline{\ESL})+ \Pr(\exists u\in [n],\;\dist(u,v)\in (k^*+\ell^*,\infty), E,\overline{\ESL})\\
 = & O(n^{-(1+\delta)})+\Pr(\exists u\in [n],\;\dist(u,v)\in (k^*+\ell^*,\infty), E,\overline{\ESL})\, ,
\end{align*}
Now let $\cH$ be the set of graphs $H$ with $v \in V(H)$ such that $\Pr(D[N^-_{\leq k_0}(v)]=H,E, \overline{\ESL})>0$. Then
\begin{align*}
& \Pr(\exists u\in [n],\;\dist(u,v)\in (k^*+\ell^*,\infty), E,\overline{\ESL})\\
\le & \sup_{H\in \cH} \Pr(\exists u\in [n],\;\dist(u,v)\in (k^*+\ell^*,\infty), \overline{\ESL} \mid D[N^-_{\leq k_0}(v)]=H)\\
\le & \sup_{H\in \cH} \sum_{u\in [n]} \Pr(\dist(u,v)\in (k^*+\ell^*,\infty), \overline{\ESL}\mid D[N^-_{\leq k_0}(v)]=H)\;.
\end{align*}
For each $H \in \cH$ there is a (non-random) constant $k=k(H)$ such that if $D[N^-_{\leq k_0}(v)]=H$ then $k_0=k(H)$, so the events
$D[N^-_{\leq k_0}(v)]=H$ and $D[N^-_{\leq k}(v)]=H$ are identical.
Furthermore, given that $D[N^-_{\leq k}(v)]=H$ we either have $d^-_{k}(v)=0$ or $d^-_{k}(v) \ge \log^4 n$.
In the latter, if $D[N^-_{\leq k}(v)]=H$ then $E_k$ occurs, so $\Pr(D[N^-_{\leq k}(v)]=H,E_k,\overline{\ESL}) > 0$,
so we can apply Proposition~\ref{prop:direct} (with $k=k(H)=k_0$). In both cases we obtain
\begin{align*}
&\Pr(\dist(u,v)\in (k^*+\ell^*,\infty),\overline{\ESL}\mid D[N^-_{\leq k_0}(v)]=H)\\
=&\Pr(\dist(u,v)\in (k^*+\ell^*,\infty),\overline{\ESL}\mid D[N^-_{\leq k}(v)]=H) \\
=&O(n^{-3})\, ,
\end{align*}
and conclude that
\begin{align*}
&\Pr(\exists u,v\in [n],\;\dist(u,v) \in(k^*+\ell^*,\infty))\\
=& O(n^{-1}) +\sum_{v\in[n]} O(n^{-(1+\delta)})+ \sum_{u,v \in[n]} O(n^{-3}) = O(n^{-\delta})\, .
\end{align*}
It follows that with high probability $\diam (D)\leq k^*+\ell^*\leq (1+\eta_r+\eps)\log_r{n}$.
\end{proof}

\section{Breadth-first search and conditioning}\label{sec:bfs_and_cond}
In this section we describe the breadth-first search (BFS) procedures, which are fundamental to our analysis, and use them to prove a handful of stochastic domination results for neighbourhood sizes in $D(n,r)$.

\subsection{Outward and inward breadth-first search}\label{sub:bfs}
Fix a digraph $D$ together with an ordering of its vertices $V(D)$ as $(v_1,\ldots,v_n)$. The {\em outward breadth-first search} (oBFS) starting from node $v \in V(D)$ is a deterministic process $((R_i^+(D,v),S_i^+(D,v)), i \ge 0)$, defined as follows. At time $i$, $R_i^+=R_i^+(D,v)$ is the set of explored vertices and $S_i^+=S_i^+(D,v)$ is the {\em sequence} of discovered but not yet explored vertices; $S$ is treated as a first-in first-out queue. Node $w \in V(D)$ has {\em index} $j$ if $w=v_j$.

Begin with $R_0^+=\emptyset$ and $S_0^+=(v)$. Now fix $i \ge 0$ and suppose $(R_i^+,S_i^+)$ are already defined.
Step $i$ of the process is defined as follows. If $S_i^+ = (s_{i,1},\ldots,s_{i,j})$ has positive length then write $u_i^+=u_i^+(D,v)=s_{i,1}$ and $C_i^+(D,v)=N^+(D,u_i^+)\setminus (R_i^+ \cup S_i^+)$. List the elements of $C_i^+(D,v)$ in increasing order of index as $w_{i,1},\ldots,w_{i,k}$; it is possible that $k=0$. Then set
\[
R^+_{i+1} = R^+_{i} \cup \{s_{i,1}\},\quad \mbox{ and } S_{i+1}^+ = (s_{i,2},\ldots,s_{i,j},w_{i,1},\ldots,w_{i,k}).
\]
In words, at step $i$, $u_i^+=s_{i,1}$ is explored, and $w_{i,1},\ldots,w_{i,k}$ are discovered and added to the back of the queue for later exploration. If $S_i^+$ has zero length (i.e., $S_i^+ = ()$), then $S_{i+1}^+=S_i^+$ and $R_{i+1}^+=R_i^+$.

Writing $i^+=i^+(D,v) = \min\{i: S_{i+1}^+=S_i^+\}$, then $R_{i^+}^+(D,v)$ is precisely the set of vertices $w$ with $\dist_D(v,w) < \infty$. The {\em oBFS tree} $T^+(D,v)$ has root $v$ and vertices $R_{i^+}(D,v)$; the children of $s_{i,1}$ are precisely the vertices
$w_{i,1},\ldots,w_{i,k}$ newly discovered in step $i$. We write $T^+(D,v,m)$ for the subtree of $T^+(D,v)$ with vertices $R^+_m(D,v) \cup S^+_m(D,v)$. Note that if $S^+_m(D,v)$ has length $\ell$ then its elements are precisely $(u^+_{m+i}(D,v),0 \le i < \ell)$, because oBFS explores these vertices before any others. We therefore have $R^+_m(D,v) \cup S^+_m(D,v) = \{u_i^+(D,v),0 \le i < m+\ell\}$,

In the {\em inward breadth-first search} (iBFS) process $((R_i^-(D,v),S_i^-(D,v)), i \ge 0)$, the sets $C_i^-(D,v)$ and the terminal time $i^-=i^-(D,v)$ are defined in just the same manner but exploring in-neighbourhoods rather than out-neighbourhoods to discover vertices; in particular $R_{i^-}^-(D,v)= \{w:\dist_D(w,v) < \infty\}$. We also write $T^-(D,v,m)$ for the subtree of $T^-(D,v)$ with vertices $R^-_m(D,v) \cup S^-_m(D,v)$.

Observe that using the notation from Section~\ref{sub:trees}, we have $T^+_k(D,v) = N^+_k(D,v)$ and $T^-_k(D,v) = N^-_k(D,v)$ for all $k$.

\subsection{Conditioning on neighbourhoods and the BFS exploration in \texorpdfstring{$D(n,r)$}{D(n,r)}}\label{sub:bfs_dnr}
We next describe the effect of iBFS on the law of $D(n,r)$. For the remainder of the section we write $D=D(n,r)$ and fix $v \in [n]$.
We write $N_i^-=N_i^-(D,v)$, $u_i^-=u_i^-(D,v)$, $T^-(m) = T^-(D,v,m)$, etcetera. Informally, the point of this section may be summarized as follows: if vertex $u_i^-$ is discovered at step $j$ then all we know about $u_i^-$ is that it has an edge to $u_j^-$ and has no edges to $u_k$ for $k < j$. We now state and prove some useful stochastic identities and inequalities which result from this.

\begin{lemma}\label{lem:ibfs_stoc}
Fix $m \in \N_0$. Conditional on $((R^-_i,S^-_i),0 \le i \le m)$, independently for all $w \in [n]$ we have
\begin{align*}
|E(w,R^-_m \cup S^-_m)| & \eqdist \Bin\left(r,\frac{|S^-_m|}{n-m}\right)
\end{align*}
if $w\not \in R_m^- \cup S_m^-$,
\[
1+\Bin\left(r-1,\frac{|S^-_m|}{n-m}\right) \preceq |E(w,R^-_m \cup S^-_m)| \preceq 1+\Bin\left(r-1,\frac{m+|S^-_m|}{n}\right)
\]
if $w \in R^-_m \cup S^-_m\setminus \{v\}$, and $|E(v,R^-_m \cup S^-_m)|  \eqdist \Bin\left(r,\frac{m+|S^-_m|}{n}\right)$.
\end{lemma}
\begin{proof}
Recall the canonical construction of $D=D(n,r)$ from the introduction, and for each $x,y \in [n]$ let $\rho(x,y) = \min\{q: L_{x,q}=y\}$. Then $\rho(x,y) \le r$ precisely if there is a copy of the oriented edge $xy$ in $D(n,r)$, and otherwise $\rho(x,y)=\infty$.

Now fix $w \in [n]$. Suppose $w \not \in R^-_m \cup S^-_m$ ; then there are
no edges from $w$ to $R_m^-$. In other words, for each $1 \le j \le r$, we have $L_{w,j} \not \in R_m^- = \{u_i^-,0 \le i < m\}$.
It follows that the conditional law of $|E(w,R_m^- \cup S_m^-)|$ given $(R^-_i,0 \le i \le m)$ and $(S^-_i,0 \le i \le m)$
is $\Bin(r, |S_m^-|/(n-|R^-_m|))$. The result follows in this case since $|R^-_m|=m$.

Now suppose $w \in R^-_m \cup S^-_m \setminus \{v\}$; then the parent $p_{T^-(m)}(w)$ lies in $R_m^-$ so satisfies $p_{T^-(m)}(w) = u_j^-$ for some $0 \le j < m$.
We have $p_{T^-(m)}(w) = u_j^-$ precisely if $w$ has no edges to $\{u_i^-: 0 \le i < j\}$ but has an edge to $u_j^-$; equivalently, $\rho(w,u_i^-)=\infty$ for each $0 \le i < j$, and $\rho(w,u_j^-)=k$ for some $1 \le k \le r$.
The heads of the first $(k-1)$ out-edges from $w$ are then uniformly distributed over $[n] \setminus \{u_i^-: 0 \le i < j\}$, and the heads
of the $r-k$ last out-edges from $w$ are uniformly distributed over $[n] \setminus \{u_i^-: 0 \le i \le j\}$.

The index $j$ is determined by $((R^-_i,S^-_i),0 \le i \le m)$. Given that $w \in R^-_m \cup S^-_m \setminus \{v\}$, we thus have
\[
1+\Bin\left(r-1,\frac{|R_m^- \cup S_m^-|-(j+1)}{n-(j+1)}\right) \preceq |E(w,R_m^- \cup S_m^-)| \preceq 1+\Bin\left(r-1,\frac{|R_m^- \cup S_m^-|-j}{n-j} \right).
\]
Since $0 \le j < m$, and $|R_m^-|=m$, the second claim follows. The argument when $w=v$ is similar but easier.
Finally, the independence asserted by the lemma follows from the independence of the random variables $(L_{w,p}:(w,p) \in [n]\times[r])$.
\end{proof}
For the next corollary, recall that $d_{\le j}^-=|N_{\le j}^-|=|N_{\le j}^-(D,v)|$. 
\begin{corollary}\label{cor:ibfs_stoc}
Fix $j \in \N_0$. Conditional on $(N_i^-,0 \le i \le j)$, independently for all $w \in [n]$ we have
$|E(w,N_{\le j}^-)| \eqdist \Bin(r,d^-_{\le j}/n)$ if $w=v$,
$|E(w,N_{\le j}^-)| \eqdist \Bin(r,d_{j}^-/(n-d^-_{\le j-1}))$ if $w \not\in N^-_{\le j}$, and
\[
1+\Bin\left(r-1,\frac{d^-_j}{n-d^-_{\le j-1}}\right) \preceq |E(w,N_{\le j}^-)| \preceq 1+\Bin\left(r-1,\frac{d^-_{\le j}}{n}\right)
\]
if $w \in N^-_{\le j} \setminus \{v\}$.
\end{corollary}
\begin{proof}
Apply Lemma~\ref{lem:ibfs_stoc} at time $m = d^-_{\le j-1}$. 
\end{proof}
\begin{corollary}\label{cor:binom}
For all $j, q, p \in \mathbb{N}_0$, given that $d^-_{j}=q$ and $d^-_{\leq j}=p$,
$d^-_{j+1} \eqdist \Bin(n-p,1-(1-q/(n-p))^r)$ and
$d^-_{j+1} \preceq \Bin\left(r(n-p),\frac{q}{n-p+q}\right)$.
\end{corollary}
\begin{proof}
If $w \in [n] \setminus N^-_{\le j}$ then $E(w,N_{\le j}^-)=E(w,N_{j}^-)$.
By Corollary~\ref{cor:ibfs_stoc},
the number of edges from $w$ to $N^-_j$ then has conditional law
$\Bin(r,q/(n-p))$, so is non-zero with probability $1-(1-q/(n-p))^r$.
The distributional identity follows since $|[n] \setminus N^-_{\le j}|= n-d^-_{\le j} =n-p$.

Next, note that $\Bin(m,1-(1-x)^r)$ is stochastically dominated by $\Bin(rm,x)$. To
see this, observe that the former is the number of columns containing at least
one 1 in an $r\times m$ matrix whose entries are iid $\Ber(x)$ random
variables, while the latter is the law of the number of ones
in such a matrix. The pigeonhole principle then yields the second claim of the lemma.
\end{proof}
Recall that $C^-_i(D,v) = N^-(D,u_i^-) \setminus (R_i^- \cup S_i^-)$ is the set of vertices discovered at step $i$ of iBFS.
\begin{corollary}\label{cor:binom2}
Fix $i,q \in \{0,1,\ldots,n\}$. Conditioned on $|S_i^-(D,v)|=q$, we have
\[
|C^-_i(D,v)| \eqdist \Bin\left(n-i-q,1-\left(1-\frac{1}{n-i}\right)^r\right)\, .
\]
\end{corollary}
We omit the proof since it is very similar to those given above.

The next lemma formalizes the intuitively clear picture that it is unlikely for the early stages of iBFS to encounter a fixed, small
subgraph of $D$ not containing the starting vertex.
\begin{lemma}\label{lem:bfs_condition}
Fix a digraph $G$ with $V(G) \subset [n]$ and $v \not \in V(G)$. Fix $s \in \N$ and let $i_0 = \inf\{i: |R^-_i \cup S^-_i|> s\}$.
Then
\[
\Pr((R^-_{i_0-1} \cup S^-_{i_0-1})\cap V(G) \ne \emptyset \mid D[V(G)]=G) \le \frac{r(s+1)|V(G)|}{n-|V(G)|-s}\, .
\]
\end{lemma}
\begin{proof}
Since $|R^-_i|=i$, we clearly have $i_0 \le s+1$. Writing $\tau=\inf \{t: (R^-_i \cup S^-_i) \cap V(G) \ne \emptyset\}$,
the probability we aim to bound is thus at most
\begin{align*}
\Pr(\tau < i_0\mid  D[V(G)]=G) & = \sum_{i=0}^{s} \Pr(\tau= i,i_0 > i\mid D[V(G)]=G) \\
						& \le \sum_{i=0}^s \Pr(\tau=i\mid i_0 > i, \tau \ge i, D[V(G)]=G)
\end{align*}
Given that $D[V(G)]=G$, there are $r|V(G)|-|E(G)|$ edges from $V(G)$ to $[n]\setminus V(G)$; the heads of such edges are uniformly distributed over $[n]\setminus V(G)$.
For $i > 0$, given that $(R^-_{i-1} \cup S^-_{i-1}) \cap V(G)=\emptyset$, the heads of these edges are uniformly distributed over $[n]\setminus (V(G) \cup R^-_{i-1}\cup S^-_{i-1})$.
Thus,
\begin{align*}
&\Pr((R^-_{i} \cup S^-_{i}) \cap V(G)\ne \emptyset \mid R^-_{i-1},S^-_{i-1},D[V(G)]=G,(R^-_{i-1} \cup S^-_{i-1}) \cap V(G)=\emptyset) \\
\le& \frac{r|V(G)|}{n-|V(G)|-|R^-_{i-1} \cup S^-_{i-1}|}\, ,
\end{align*}
since, in this situation, the only way that $(R^-_{i} \cup S^-_{i}) \cap V(G)\neq \emptyset$ is in the case that some edge with tail in $V(G)$ has as a head the vertex $u^-_{i-1}$.

Given that $i_0 > i$ and $\tau \ge i$ we indeed have $(R^-_{i-1} \cup S^-_{i-1}) \cap V(G)=\emptyset$, and also have $|R^-_i \cup S^-_i| \le s$,
so $\Pr(\tau=i\mid i_0 > i, \tau \ge i, D[V(G)]=G) \le r|V(G)|/(n-|V(G)|-s)$. Using this bound in the above sum, the result follows.
\end{proof}

Finally, we require the following, rather simple result for oBFS.
\begin{lemma}\label{lem:obfs_stoc}
Fix $m \in \N_0$. Conditional on $R^+_m$ and $S^+_m$, independently for all $w \in [n]\setminus R^+_m$ we have
\[|E(w,R^+_m \cup S^+_m)| \preceq \Bin(r,(|R^+_m|+|S^+_m|)/n) \preceq \Bin(r,(rm+1)/n)\, .\]
\end{lemma}
\begin{proof}
We omit the proof of the first inequality, which parallels that of Lemma~\ref{lem:ibfs_stoc}.
For the second, note that $|R^+_m \cup S^+_m| \le rm+1$ since $D$ is $r$-out regular.
\end{proof}

\section{In-neighbourhoods: technical lemmas}\label{sec:lemmas}
In this section we gather a few basic estimates that describe the size and structure of in-neighbourhoods of vertices in $D=D(n,r)$.

The following result controls the growth of the in-neighbourhoods in $D(n,r)$.
\begin{proposition}\label{prop:UB_in_neighs}
For all $\alpha>0$,
$$
\Pr\left(\exists v\in [n],\; \exists j\geq 0 \text{ such that } d^-_j(v) \geq (r+\alpha)^j \log_r^2{n}\right) =O(n^{-4})\;.
$$
\end{proposition}
\begin{proof}
Fix $v\in [n]$. We prove that
\begin{align}\label{eq:prob1}
\Pr\left(\exists j\geq 0 \text{ such that } d^-_j(v) \geq (r+\alpha)^j \log_r^2{n}\right) &=O(n^{-5})\, ;
\end{align}
a union bound over $v\in [n]$ then proves the proposition.

For $j\geq 0$, let $E_j = \{d^-_{j}(v)< (r+\alpha)^{j}\log_r^2{n}\}$. Then
\begin{align}\label{eq:prob2}
\Pr\left(\exists j\geq 0 \text{ such that } d^-_j(v) \geq (r+\alpha)^j \log_r^2{n}\right)  =
\Pr\left(\cup_{j=1}^n \overline{E_j}\right) \leq \sum_{j=1}^n \Pr(\overline{E_j}\mid \cap_{j'<j} E_{j'})\;.
\end{align}

By Corollary~\ref{cor:binom}, for every $p\geq q$ and every $a\geq 0$,
\begin{equation*}
\Pr(d^-_{j}(v)\geq a\mid d^-_{j-1}(v)=q, d^-_{\leq j-1}(v)=p) \leq \Pr\left(\Bin\left(r(n-p),\frac{q}{n-p+q}\right)\geq a\right) \enspace.
\end{equation*}
Note that $r(n-p)\cdot\frac{q}{n-p+q}\leq rq$. Set $a=(r+\alpha)^{j}\log_r^2{n}$, and observe that if $E_{j-1}$ occurs then $d^-_{j-1}(v)=q \leq q_0:=(r+\alpha)^{j-1}\log_r^2{n}$.
Finally, for such $q$ we have $a = (r+\alpha) q_0 \ge (r+\alpha)q$, so
\begin{align*}
\Pr(\overline{E_j}\mid \cap_{j'<j} E_{j'}) &= \Pr(d^-_{j}(v)\geq a\mid \cap_{j'<j} E_{j'}) \\
&\leq \sup_{q\leq q_0,p} \Pr(d^-_{j}(v)\geq a\mid d^-_{j-1}(v)=q, d^-_{\leq j-1}(v)=p)\\
&\le \sup_{q\leq q_0,p} \Pr(d^-_{j}(v)\geq rq+\alpha q\mid d^-_{j-1}(v)=q, d^-_{\leq j-1}(v)=p)\\
&\leq \sup_{q\leq q_0} e^{-\frac{\alpha^2 q^2}{2(r+\alpha/3)q}} = e^{-\Omega(\log_r^2{n})}=O(n^{-6}) \enspace,
\end{align*}
where we used the Chernoff bound~\eqref{eq:cher1}.
Using this bound in~\eqref{eq:prob2} proves~\eqref{eq:prob1}.
%
\end{proof}

The next lemma controls the probability that the sequence $(d_k^-(v),k \ge 1)$ exhibits a large decrease in value for relatively small values of $k$.
\begin{lemma}\label{lem:no_strange_things}
Fix $v\in [n]$. Uniformly in $k\leq 0.99\log_r{n}$ and $\omega\geq \log_r^3{n}$, we have
\[\Pr(d^-_{\leq k}(v)\geq \omega^2,\,  d^-_{k}(v)\leq  \omega )=O(n^{-3}).
\]
\end{lemma}
\begin{proof}
Fix $k$ and $\omega$ as above, and $v \in [n]$.
Let $\tau=\min\{j:\; d^-_{j}(v)\geq \omega^2/k\}$. If $d^-_{\leq k}(v)\geq \omega^2$ and $d^-_{k}(v)\leq \omega$, then $\tau< k$, so
\begin{align}
\Pr(d^-_{\leq k}(v)\geq \omega^2, d^-_{k}(v)\leq  \omega)
&= \Pr(d^-_{\leq k}(v)\geq \omega^2,d^-_{k}(v)\leq \omega,\tau < k )\nonumber\\
&= \sum_{j=1}^{k-1} \Pr(d^-_{\leq k}(v)\geq \omega^2, d^-_{k}(v)\leq \omega, \tau =j )\nonumber\\
&\leq \sum_{j=1}^{k-1} \Pr(d^-_{j}(v)\geq \omega^2/k, d^-_{k}(v)\leq  \omega)\;.\label{eq:strange_1}
\end{align}
Let $\sigma=\min \{i \ge \tau:\; d^-_{i+1}(v)\leq d^-_{i}(v)\}$. For any $j< k$, if $d^-_{j}(v)\geq \omega^2/k$ and $d^-_{k}(v)\leq  \omega$, then $\sigma< k$, so
\begin{align}
\Pr(d^-_{j}(v)\geq \omega^2/k , d^-_{k}(v)\leq  \omega)&\leq \Pr(d^-_{j}(v)\geq \omega^2/k,\; \sigma< k)\;.\label{eq:strange_3}
\end{align}
Now fix $\alpha>0$ small enough that $(r+\alpha)^k \log_r^2 n < n^{0.999}$ for $n$ large; this is possible by our choice of $k$.
Also, for $\ell \in \N$ let $E_{\ell} = \{\forall i \le \ell, d^-_i(v) < (r+\alpha)^i \log_r^2{n}\}$, and let
$E = \bigcap_{\ell \ge 1} E_{\ell}$.
By Proposition~\ref{prop:UB_in_neighs}, we have $\Pr(\overline{E})=O(n^{-4})$, so for all $1\leq j\leq k-1$,
\begin{align}
\Pr(d^-_{j}(v)\geq \omega^2/k,\; \sigma< k)&\leq  \Pr(\overline{E})+ \Pr(d^-_{j}(v)\geq  \omega^2/k,\; \sigma< k,\; E)\nonumber\\
 &\leq  O(n^{-4})+ \sum_{\ell=j}^{k-1}\Pr(d^-_{j}(v)\geq  \omega^2/k,\; \sigma=\ell,\; E)\nonumber\\
 &\leq  O(n^{-4})+ \sum_{\ell=j}^{k-1}\Pr(d^-_{\ell}(v)\geq  \omega^2/k,\; d^-_{\ell+1}(v)\leq d^-_{\ell}(v),\; E_{\ell})\nonumber\\
 &\leq  O(n^{-4})+\sum_{\ell=j}^{k-1}\Pr(d^-_{\ell+1}(v)\leq d^-_{\ell}(v)\mid d^-_{\ell}(v)\geq \omega^2/k,\; E_{\ell})\;.
 \label{eq:strange_2}
\end{align}
On $E_{\ell}$ we have $d^-_{\leq \ell}(v)< (r+\alpha)^{\ell+1}\log_r^2{n} \le (r+\alpha)^{k}\log_r^2{n}< n^{0.999}$.

Now fix $0 < q \le p \le n^{0.999}$ and let $X$ be distributed as $\Bin\left(n-p,1-\left(1-\frac{q}{n-p+q}\right)^r\right)$.
Using that
$(1-x)^i \le 1-ix +{i \choose 2} x^2$ for $x \in (0,1)$ and $i \in \N$, for we have
\begin{align*}
\E(X)
 &\ge (n-p) \left(\frac{rq}{n-p+q}-\frac{\binom{r}{2}q^2}{n-p+q}\right)
 = rq \cdot \frac{(n-p)(n-p-\frac{r-3}{2}q)}{(n-p+q)^2} \\
 &> rq \cdot \left( 1- \frac{3rp}{n}\right) > \frac{2rq}{3}\, ,
\end{align*}
the last inequality for $n$ large since $p\leq  n^{0.999}$.
Now write $q_0 = \omega^2/k$ and $p_0 = (r+\alpha)^{\ell+1} \log_r n < n^{0.999}$.
By Corollary~\ref{cor:binom} and the Chernoff bound~\eqref{eq:cher1b}, we have
\begin{align*}
&\Pr(d^-_{\ell+1}(v)\leq d^-_{\ell}(v)\mid d^-_{\ell}(v)\geq \omega^2/k,\; E_{\ell})\\
\leq& \sup_{q_0 \le q \le p \le p_0} \Pr(d^-_{\ell+1}(v)\leq d^-_{\ell}(v)\mid d^-_{\ell}(v)=q, d^-_{\leq \ell}(v)=p)\\
\leq& \sup_{q_0 \le q \le p \le p_0} \Pr(d^-_{\ell+1}(v)\leq q \mid d^-_{\ell}(v)=q, d^-_{\leq \ell}(v)=p)\\
\leq& \sup_{q_0 \le q \le p \le p_0} \Pr\left(X\leq q\right)\\
\leq& \sup_{q_0 \le q \le p \le p_0} \Pr\left(X \leq \E(X)-(2r/3-1)q\right)\\
\leq& \sup_{q_0 \le q \le p \le p_0} e^{-\frac{(2r/3-1)^2q^2}{2(rq+(2r/3-1)q/3)}} \le e^{-\frac{q_0}{18r+2}}\, ,
\end{align*}
in the last line using that $r \ge 2$. Finally, $q_0 = \omega^2/k \ge \log_r^2 n$, so $e^{-q_0/(18r+2)} = O(n^{-4})$.
Combining the preceding inequality with~\eqref{eq:strange_1},~\eqref{eq:strange_3} and~\eqref{eq:strange_2} yields
\begin{align*}
\Pr(d^-_{\leq k}(v)\geq \omega^2, d^-_{k}(v)\leq  \omega)
&\leq  \sum_{j=1}^{k-1}  \sum_{\ell=j}^{k-1} \Pr(d^-_{\ell+1}(v)\leq d^-_{\ell}(v)\mid d^-_{\ell}(v)\geq  \omega^2/k,\; E) + O(k n^{-4}) \\
&\leq  O(k^2n^{-4}) = O(n^{-3})\, . \qedhere
\end{align*}
\end{proof}

The next lemma compares the law of $T^-(v)$ to that of a Poisson$(r)$ Galton-Watson tree.
A similar result, in the setting of undirected graphs, can be found in~\citep[Lemma 2.2]{riordan}.
\begin{lemma}\label{lem:tree_like}
Let $r\geq 2$ and let $T'$ be a $v$-rooted plane directed tree where all edges point to
the root. Suppose that $|V(T')|\leq n/2$. Then, for any $k \geq 0$ we have
\begin{equation*}
\Pr(T_{\leq k}^-(D,v) \cong T')= e^{O(|V(T')|^2/n)}\Pr(\cT_{\leq k} \cong T')\;.
\end{equation*}
where $\cT$ is a Galton-Watson branching tree whose offspring is Poisson with parameter $r$.
\end{lemma}
\begin{proof}
Fix $k \in \N_0$ and a plane tree $T'$ of height at most $k$, and write $t=|V(T')|$. Recall the canonical labelling of $V(T')$ with labels from $\emptyset \cup \bigcup_{i \ge 1} \N^i$ introduced in Section~\ref{sub:trees}. Consider the iBFS procedure on $T'$ started at its root $v$. To make sense of this, we must specify the order in which the children of a vertex $u$ are added to the set of discovered vertices. We use the left-to-right order: so if $u$ is explored at step $i$ (i.e. $u_i^-(T',v)=u$) then the rightmost child of $u$ is the last element of $S_i^-(T',v)$.

Let $a_i = |C_i^-(T',v)|$ be the number of children of $u_i^-(T',v)$, and let $s = |V(T'_{\leq k-1})|$ be the number of vertices of $T'$ at distance at most $k-1$ from the root. In order to check if $T_{\leq k}(D,v)$ and $T'$ are isomorphic, it suffices to perform $s$ steps of the iBFS exploration from $v$ in $D$. We then have
\begin{align*}
\Pr(T_{\leq k}^-(D,v) \cong T') &= \Pr\left(|C^-_i(D,v)|=a_{i}, 0 \le i <s \right)\\
&= \prod_{i=0}^{s-1} \Pr\left(|C^-_i(D,v)|=a_{i}\mid |C^-_j(D,v)|=a_{j},0 \le j < i \right)\\
&= \prod_{i=0}^{s-1} \Pr\left(|C^-_i(D,v)|=a_{i}\mid  |S_{i}^-(D,v)|=1+\sum_{j=0}^{i-1} (a_j-1) \right)\;,
\end{align*}
where the last line is due to the symmetry of the model.

Writing $q_i=1+\sum_{j=0}^{i-1} (a_j -1)$, by Corollary~\ref{cor:binom2} we then have
\begin{align*}
& \Pr(|C^-_i(D,v)|=a_{i} \mid |S_{i}^-(D,v)|=q_{i}) \\
=&  \binom{n-i -q_{i}}{a_i}
\left(1-\left(1-\frac{1}{n-i}\right)^r\right)^{a_i}
\left(1-\frac{1}{n-i}\right)^{r(n-q_{i}-a_i-i)}\\
=& e^{O\left(\frac{a_i(a_i+r)}{n-i-q_i}\right)} \frac{(n-i-q_{i})^{a_i}}{a_i!}
\left(\frac{r}{n-i}\right)^{a_i}
\left(1-\frac{1}{n-i}\right)^{r(n-q_{i}-a_i-i)}
\enspace.
\end{align*}

Now let $\cT$ be a Poisson$(r)$ Galton-Watson tree; write $\rho$ for the root of $\cT$. Build $\cT$ via iBFS starting from $\rho$. In this manner, we may couple $\cT$ with a sequence $(\xi_i,i \ge 0)$ of iid $\Po(r)$ random variables so that for $0 \le i < |V(\cT)|$ we have $|C^-_i(\cT,\rho)|=\xi_i$. It follows that
$$
\Pr(\cT_{\leq k} \cong T') = \Pr\left(\cap_{i=0}^{s-1} \; \xi_{i}=a_{i}\right) = \prod_{i=0}^{s-1} \Pr(\xi_{i}=a_{i})=
\prod_{i=0}^{s-1} e^{-r} \frac{r^{a_i}}{a_i!}\;.
$$

Using that $1+x\le e^x$, this gives
\begin{align*}
\frac{ \Pr(|C^-_i(D,v)|=a_{i} \mid |S_{i}^-(D,v)|=q_{i}) }{\Pr(\xi_i=a_i)} &=e^{O\left(\frac{a_i(a_i+r)}{n-i-q_i}\right)}
\frac{(n-i-q_{i})^{a_i}}{(n-i)^{a_i}}
\cdot \frac{\left(1-\frac{1}{n-i}\right)^{r(n-q_{i}-a_i-i)}}{e^{-r}}\\
&= e^{O\left(\frac{q_{i}+a^2_i}{n -i-q_i}\right)}\;.
\end{align*}
Since $i+q_{i} \leq  t \leq n/2$, $\sum_{i=0}^{s-1} a^2_i\leq t^2$, and $s \leq t$, we have
$$
\sum_{i=0}^{s-1} \frac{q_{i}+a^2_i}{n-i-q_i}\leq \sum_{i=0}^{s-1} \frac{ t+a_i^2}{n/2}=O\left(\frac{t^2}{n}\right)\;. $$
It follows that
\begin{equation*}
\frac{\Pr(T_{\leq k}^-(D,v) \cong T')}{\Pr(\cT_{\leq k} \cong T')} =
\prod_{i=0}^{s-1}
\frac{ \Pr(|C^-_i(D,v)|=a_{i} \mid |S_{i}^-(D,v)|=q_{i}) }{\Pr(\xi_i=a_i)}
=e^{O(t^2/n)}\, .\qedhere
\end{equation*}
\end{proof}

\begin{lemma}[\cite{riordan}, Lemma 2.1]\label{lem:riordan}
Let $\cT$ be a Poisson$(r)$ Galton-Watson tree. 
There exist constants $c,C>0$ such that for every $\omega\geq 2$ and $k\geq 1$ we have
\begin{equation*}
c\cdot \min\{(r(1-\lambda_{r}))^{k-k'},1\}\leq \Pr(0<|\cT_k|<\omega)\leq C
(r(1-\lambda_{r}))^{k-k'}\;,
\end{equation*}
where $k'=\lfloor\log_r \omega \rfloor$.
\end{lemma}
Recall that the probability of survival in $\cT$ is $\Pr\left(\sum_{k\geq 0} |\cT_k|=\infty\right) \in (0,1)$. Essentially, the preceding lemma states that given the branching process survives for the first $k$ generations, the probability that $|\cT_k|< \omega$ decays exponentially in $k$ (provided that $\omega$ is small enough with respect to $k$).
The final and principal result of this section is to prove a corresponding bound with $d^-_k(v)$ in place of $|\cT_k|$.
\begin{proposition}\label{lem:towers}
For every $v\in [n]$, $k\leq 0.99 \log_r{n}$ and $\log^3_r{n} \leq \omega \leq n^{1/6}$
\begin{align*}
\Pr(0< d^-_k(v) < \omega)= (1+o(1))\Pr(0<|\cT_k|<\omega)+O(n^{-3})\;.
\end{align*}
\end{proposition}
\begin{proof}
Write $Z_k = |\cT_k|$. We first prove an upper bound on $\Pr(0< d^-_k(v) < \omega)$.
By Lemma~\ref{lem:tree_like},
\begin{align*}
\Pr(0< d^-_k(v) < \omega)
&=
\sum_{\{T':|T'_k|\in (0,\omega)\}} \Pr\left(T_{\leq k}(D,v)\cong T'\right)
\\ &\leq
\sum_{\{T':|T'_k|\in (0,\omega)\atop |T'|\leq n^{1/3}\}}
e^{O\left(\frac{|V(T')|^2}{n}\right)}
\Pr(\cT_{\leq k}\cong T')
+
\Pr(d^-_{\leq k}(v) \geq \omega^2, d^-_k(v) \in (0,\omega))
\\ &\leq
(1+o(1))\Pr(0<Z_k<\omega)+O(n^{-3})\;,
\end{align*}
where in the last inequality we used Lemma~\ref{lem:no_strange_things}.

We now turn to the lower bound. A similar argument to that above gives
\[
\Pr(0< d^-_k(v) < \omega) \ge (1+o(1))\Pr(0<Z_k<\omega) - \Pr(\sum_{j=0}^k Z_j \ge \omega^2,Z_k < \omega).
\]
Bounding the second probability is straightforward. First, fix $j < k$ and $a,b \in \N$. Given that $Z_j=a$, by the branching property, each of the $a$ subtrees of $\cT$ rooted at a node in $\cT_{j}$ survives independently with probability $p:= \Pr(|\cT|=\infty)$. Note that $p > 0$ is independent of $n$.
But $Z_k$ is at least the number of such subtrees which survive, so $\Pr(Z_k<b\mid Z_j=a) \le \Pr(\Bin(a,p) < b)$.

Finally, if $\sum_{j=0}^k Z_j \ge \omega^2$ then $\max_{0 \le j \le k} Z_j \ge \omega^2/(k+1)$. In other words, letting $j_0=\inf\{j: Z_j \ge \omega^2/(k+1)\}$, we must have $j_0 \le k$. It follows from the preceding paragraph (conditioning on the value of $j_0 \le k$) that
\[
\Pr(\sum_{j=0}^k Z_j \ge \omega^2,Z_k < \omega) \le \Pr(\Bin(\omega^2/(k+1),p) < \omega) = O(n^{-3}),
\]
the final inequality by a Chernoff bound since $\omega^2/(k+1) \ge \omega \log^2 n$. The proposition follows.
\end{proof}



\section{Out-neighbourhoods: technical lemmas}\label{sec:out_tech}
Recall that $\ESL$ is the event that $D$ contains no loop vertices. As in the statement of Proposition~\ref{prop:direct} we now fix $k \in [n]$, let $E =\{d^-_k(v)\geq \log^4{n},d^-_{\le k}(v) \le \log^7 n\}$, and fix a graph $H$ with $v \in V(H) \subset [n]$ such that $\Pr(D[N^-_{\leq k}(v)]=H,E, \overline{\ESL})>0$.
It is useful to write $B = N_k^-(H,v)$; note that this is a deterministic set since $H$ is deterministic, and on $D[N^-_{\le k}]=H$
we have $N^-_k(v) = B$. Let $\hat{n} = n-|V(H)|+|B|$. By the assumptions on $H$ we have $\hat{n} \ge n - \log^7 n + \log^4 n$.

For any event $A$, write $\Pr^{H}(A)=\Pr(A \mid  D[N^-_{\le k}]=H)$. The following fact describes the distribution of $D$ under $\Pr^H$. Its proof follows from straightforward considerations and is omitted.
\begin{fact}
Given that $D[N^-_{\le k}]=H$, the conditional distribution of $D(n,r)$ is that of the graph $\hat{D}$ defined as follows.
First, $\hat{D}[V(H)]=H$.

Next, independently for each $w\notin V(H)$, let $\hat{L}_w=(\hat{L}_{w,1},\dots, \hat{L}_{w,r})$ be a vector chosen uniformly at random from the $\hat{n}^r$ vectors $(s_1,\dots, s_r)\in (([n]\setminus V(H))\cup B)^r$. Then for each $ i\in [r]$ add a directed edge from $w$ to $\hat{L}_{w,i}$.

Finally, for $w \in V(H)$, let $t_w = r-|E_H(w,H)|$; this is the number of edges with tail $w$ and head not in $H$. Independently for each $w \in V(H)$, 
let $\hat{L}_w = (\hat{L}_{w,1},\dots, \hat{L}_{w,t_w})$ 
be a vector chosen uniformly at random from 
$([n]\setminus V(H))^{t_w}$, and for each $ i\in [t_w]$ add a directed edge from $w$ to $\hat{L}_{w,i}$. 
\end{fact}
For the remainder of the section, fix $u \in [n]$ and write $N^*_j=N^+_j(u)\setminus V(H)$, $d^*_j=|N^*_j|$. We continue with a simple lemma.


\begin{lemma}\label{lem:dfive}
We have $\Pr^H(d_5^* < 5,N^+_{\le 5}(u) \cap V(H)=\emptyset,\overline{\ESL}) =O(n^{-3})$
\end{lemma}
\begin{proof}
If $N^+_{\le 5}(u) \cap V(H)=\emptyset$ then $d^*_5 = d^+_5$. In this case,
since $r \ge 2$, it is a simple combinatorial exercise to check that if also $d^*_5<5$ then $D[N^*_{\leq 5}]$ has at least two more edges than vertices. For any fixed digraph $\widehat{D}$ with at least two more edges than vertices and with no self loops, it is easily seen that $\Pr^H(D[N^+_{\leq 5}]\cong \hat{D},N^+_{\le 5}(u) \cap V(H)=\emptyset) = O(n^{-3})$. (This is not true for digraphs with self-loops if $r=2$; the probability $v$ itself is a loop vertex is $O(n^{-2})$ and in this case $d^*_5=0$.) The number of isomorphism classes of digraphs with diameter at most $5$ and maximum out-degree $r$ is bounded, and the result follows.
\end{proof}
We next show that with high probability, each generation $N^*_j$ is approximately $r$ times larger than the last, until $j$ is nearly $(\log^r n)/2$.

\begin{lemma}\label{lem:grow_by_r}
Let $\sigma = \inf\{i\ge 5: d_{i}^* < r^{i-5}+5\}$. Then $\Pr^H(5 < \sigma \le (\log_r n)/8) =O(n^{-3})$.
\end{lemma}
\begin{proof}
Fix $j > 5$.
If $d^*_5\geq 5$ and $d^*_{i+1}\geq r d^*_i -4$ for every $5 \le i < j$, then by induction
$d^*_j \ge r^{j-5}+5$.
We thus have
\[
\Pr^H(5 < \sigma \le j) \le \Pr^H\left(d^*_5 \ge 5, \bigcup_{i=5}^{j} \{|d^*_{i}| < r^{i-5}+5\}\right) \le \sum_{i=5}^{j-1} \Pr(d^*_{i+1} \leq  rd^*_i-4).
\]
Now fix $5 \le i < j$. Condition on $N^*_{\le i}$, and recall that the random variables
$\{L_{w,m} : w \in N^*_{i}, m \in [r]\}$ are the heads of edges from vertices in $N^*_i$.
Reveal the values of these random variables one-at-a-time; say a conflict occurs if $L_{w,m} \in N^*_{\le i} \cup V(H)$ or
$L_{w,m}=L_{w',m'}$ for a previously revealed $L_{w',m'}$. If $d^*_{i+1}\leq rd^*_i - 4$ then at least $4$ conflicts occur.

Under $\Pr^H$, the random variables $L_{w,m}$ are independent and uniform over $([n]\setminus V(H)) \cup B$.
When $L_{w,m}$ is revealed there are less than $r^{i+2}+|B|$ locations that can cause a conflict, since $|N^*_{\le i+1}| < r^{i+2}$,
so the probability of a conflict is less than $(r^{i+2}+|B|)/\hat{n}$. The set $\{L_{w,m} : w \in N^*_{i}, m \in [r]\}$ has size at most $r^{i+1}$, and
$|B| \le \log^7 n$; it follows that
\begin{align*}
\Pr^H(d^*_{i+1} \leq r d^*_i - 4) &\le \Pr\left( \Bin(r^{i+1},(r^{i+2}+|B|)/\hat{n}) \ge 4\right) \\
&\le {r^{i+1} \choose 4} \cdot \left(\frac{r^{i+2}+\log^7 n}{\hat{n}}\right)^4 \le \frac{C(r^{8i}+\log^{28} n)}{n^4}\, ,
\end{align*}
where in the last inequality we used that $\hat{n} \ge n - \log^7 n$. For $j \le (\log_r n)/8$ we thus have
\[
\Pr^H\left(d^*_5 \ge 5, \bigcup_{i=5}^{j} \{|d^*_{i}| < r^{i-5}+5\}\right) \le \sum_{i=5}^{j-1} \frac{C(r^{8i}+\log^{28} n)}{n^4} = O(n^{-3})\, .\qedhere
\]
\end{proof}
The third lemma of the section shows that out-neighbourhoods continue to grow rapidly until they reach size close to $n/\log n$.
\begin{lemma}\label{lem:aux}
There is $C' > 0$ such that for all $i$ with $r^i \le n/\log_r n - 2\log^7 n$,
\[
\Pr^H \left(d^*_i \ge \log^3_r n, d^*_{i+1}\leq r d^*_i\cdot \left(1-\frac{2r^2}{\log_r{n}}\right)\right) \le e^{-C'\log^2 n}\, .
\]
\end{lemma}
\begin{proof}
We have
\begin{align*}
\Pr^H
&\left(d^*_i \ge \log^3_r n, d^*_{i+1}\leq r d^*_i\cdot \left(1-\frac{2r^2}{\log_r{n}}\right)\right) \\ 
\le& \sup_{a \in [\log^3_r n,r^{i+1}]}
\Pr^H\left(d^*_{i+1}\leq r d^*_i\cdot \left(1-\frac{2r^2}{\log_r{n}}\right) \mid d^*_i=a\right)\, .
\end{align*}
Condition on $N^*_{\le i}$, and reveal the random variables $\{L_{w,m} : w \in N^*_{i}, m \in [r]\}$ one-at-a-time
as in the previous proof. When $L_{w,m}$ is revealed there are less than $r^{i+2}+|B|$ locations that can cause a conflict,
so under $\Pr^H$ the probability of a conflict is at most $(r^{i+2}+|B|)/\hat{n}$.
If $d^*_{i+1}\leq rd^*_i - t$ then at least $t$ conflicts occur, so we obtain
\[
\Pr^H\left(d^*_{i+1}\leq r d^*_i\cdot \Big(1-\frac{2r^2}{\log_r{n}}\Big) \mid d^*_i=a\right)
\le \Pr\left(\Bin\Big(ar,\frac{r^{i+2}+|B|}{\hat{n}}\Big) \ge \frac{2r^3 a}{\log_r n}\right)\, .
\]
Using that $r^i \le n/\log_r n - 2\log^7 n$ and that $|B| \le \log^7 n$ and $\hat{n} \ge n - \log^7 n$, it is straightforward to verify that $ar(r^{i+2}+|B|)/\hat{n} \le r^3 a/\log_r n$.
A Chernoff bound then gives
\[
\Pr^H\left(d^*_{i+1}\leq r d^*_i\cdot \Big(1-\frac{2r^2}{\log_r{n}}\Big) \mid d^*_i=a\right)
\le e^{-3r^3a/(8\log_r n)} \le e^{-C'\log^2 n}\, ,
\]
for some constant $C'=C'(r)$. The latter inequality follows since $a \ge \log^3_r n$\, .
\end{proof}
The following is an easy consequence of the preceding lemma, and concludes the section.
\begin{corollary}\label{cor:jell}
Let $j^* = 3\log_r \log_r n + 5$ and let $\ell^* = \log_r n - \log_r \log_r n - 1$. Then there are $c,C > 0$ such that
\[\Pr^H\left(d^*_{j^*} \ge \log_r^3 n, d^*_{\ell^*} \le \frac{c n}{\log_r{n}}\right) \le e^{-C\log^2 n}\, .\]
\end{corollary}
\begin{proof}
If $d^*_{i} \ge \log^3_r n$ and $d^*_{i+1} \ge r d^*_i\cdot \left(1-\frac{2r^2}{\log_r{n}}\right)$
then $d^*_{i+1} \ge \log^3_r n$. Since $\ell^*-j^*=\log_r{n}-4\log_r\log_r n -6$ we also have
\[\left(1-\frac{2r^2}{\log_r{n}}\right)^{\ell^*-j^*} r^{\ell^*-j^*} \log^3 n
\ge \left(1-\frac{2r^2}{\log_r{n}}\right)^{\log_r{n}-1} \frac{n}{r^6 \log_r^4{n}}\cdot \log_r^3{n}
\ge  \frac{e^{-2r^2}}{r^6}\cdot \frac{n}{\log_r{n}}\;,
\]
where we used that $(1-b/x)^{x-1}\geq e^{-bx}$.
With $c=\frac{e^{-2r^2}}{2r^6}$, it follows from the preceding inequalities that if $d^*_{j^*} \ge \log^3_r n$ but
$d^*_{\ell^*} < cn/\log_r{n}$ then there is $i \in [j^*,\ell^*-1]$ such that
$d^*_{i} \ge \log^3_r n$ and $d^*_{i+1} < r d^*_i\cdot \left(1-\frac{2r^2}{\log_r{n}}\right)$. By Lemma~\ref{lem:aux}, there exists some $C'$ such that
\begin{align*}
 \Pr^H\left(d^*_{j^*} \ge \log^3_r n, d^*_{\ell^*} < \frac{cn}{\log_r{n}}\right)
& \le  \sum_{i=j^*}^{\ell^*-1}
\Pr^H \left(d^*_i \ge \log^3_r n, d^*_{i+1}\leq r d^*_i\cdot \left(1-\frac{2r^2}{\log_r{n}}\right)\right)\\
& \le (\ell^*-j^*)\cdot e^{-C'\log^2 n}\\
& \le e^{-C\log^2 n}\;, \qedhere
\end{align*}
for some constant $C< C'$.
\end{proof}

\section{Upper bound on the Diameter}\label{sec:upper}

In this section we prove Lemma~\ref{lem:inverse} and Proposition~\ref{prop:direct} from Section~\ref{sec:upper_proof}.
Throughout the section, we fix $u,v \in [n]$.

\begin{proof}[Proof of Lemma~\ref{lem:inverse}]
It suffices to prove the lemma assuming $\eps < 1/10$. Let $k^* =(\eta_r+\eps)\log_r{n}$,
and let $\delta=\frac{\eps}{2\eta_r}$. An easy computation shows that $\eta_r\leq 4/5$ for every $r\geq 2$, so $k^*\leq 0.9 \log_r{n}$. Recall the definition $k_0=k_0(v)=\min\{k:\;d^-_{k}(v)\not\in (0,\log^4{n})\}$. If $k_0 > k^*$ then $0< d^-_{k^*}(v) < \log^4{n}$, so by Lemma~\ref{lem:riordan} and Proposition~\ref{lem:towers} we have
\begin{align*}
\Pr(k_0 > k^*) & = \Pr(0< d^-_{k^*}(v) < \log^4{n})\\
&\leq (1+o(1))C\cdot (r(1-\lambda_r))^{k^*-4\log_r\log{n}}+O(n^{-3})\\
&\leq O\left( r^{(\log_r (r(1-\lambda_r)))(\eta_r+\eps+o(1))\log_r{n}}\right)+O(n^{-3})\\
&= O\left( r^{-(1+2\delta+o(1))\log_r{n}}\right)+O(n^{-3})\\
&= O\left(n^{-(1+\delta)}\right)\;,
\end{align*}
where we used that $1+\log_r{(1-\lambda_{r})}=-\eta_r^{-1}$.
For all $i < k_0$ we have $d^-_i(v) <\log^4 n$, so if $k_0 \le k^*$ then $d^-_{\le k_0-1}(v) < k^* \log^4 n \le \log^6 n$.
 In this case, for $n$ large, to have
$d^-_{\le k_0}(v) \ge \log^7 n$ we must have $d^-_{k_0} \ge (\log^7 n)/2$. It follows by Corollary~\ref{cor:binom} and a Chernoff bound
\begin{align*}
& \Pr(k_0 \le k^*, d^-_{\le k_0}(v) \ge (\log^7 n)/2) \\
\le& \sup_{k \le k^*-1}  \sup_{q< \log^4 n} \sup_{p <\log^6 n}
\Pr(d^-_{k+1}(v) \ge (\log^7 n)/2 \mid d^-_k(v)=q,d^-_{\le k}(v)=p) \\
 \le& \Pr(\Bin(rn/2,2\log^4 n/n) \ge (\log^7 n)/2) \\
 \le& e^{-(\log^7 n)/8}.
\end{align*}
Combining the two preceding bounds, the lemma follows.
\end{proof}
The proof of Proposition~\ref{prop:direct} occupies the remainder of the section.
Let $\tau = \min\{j\geq 1:\; N_j^+(u) \cap N_{\le k^-}(v) \ne \emptyset\}$, so in particular $\dist(u,v) = \tau+k$.
\begin{lemma}\label{lem:tau_tauprime}
Fix $\eps > 0$ and let $\tau' = \min\{j \ge 1: |N_j^+(u)\setminus V(H)| \ge \eps n/\log n\}$.
Then for $n$ large,
\[
\Pr^H(\tau > \tau'+1) = e^{-\eps r \log^3 n}
\]
\end{lemma}
\begin{proof}
First,
\begin{align*}
	& \Pr(D[N^-_{\le k}(v)]=H,\tau > \tau'+1) \\
 = 	&  \sum_{F} \Pr(D[N^-_{\le k}(v)]=H,D[N^+_{\le \tau'}(u)\setminus V(H)]=F,\tau > \tau'+1)\,  \\
=	& \sum_{F} \Pr(D[N^-_{\le k}(v)]=H,D[N^+_{\le \tau'}(u)\setminus V(H)]=F)  \\
&\hspace{1cm} \cdot\Pr(\tau > \tau'+1\mid D[N^-_{\le k}]=H,D[N^+_{\le \tau'}(u)\setminus V(H)]=F)\, .
\end{align*}
 where the sums are over graphs $F$ with $V(F) \cap V(H)=\emptyset$, such that $u \in V(F)$ and such that, for some $\ell > 0$,
$V(F)=\bigcup_{j=0}^\ell N_{j}^+(F,u)$ 
and $\ell = \min\{i: |N_{i}^+(F,u)| \ge \epsilon n/\log n\}$.
We now bound the final probability.
Under such conditioning, the out-edges from $N^+_{\tau'}$ are uniformly distributed over $([n]\setminus V(H)) \cup B$.
There are more than $r(\eps n/\log n)$ such out-edges; to have $\tau > \tau'+1$ the heads of such edges must all avoid $B$; so
\begin{align*}
\Pr(\tau>\tau'+1\mid D[N^-_{\le k}(v)]=H,D[N^+_{\le \tau'}(u)\setminus V(H)]=F) &\le \left(1-\frac{|B|}{\hat{n}}\right)^{r(\eps n/\log n)} \\
&\le \left(1-\frac{\log^4 n}{n}\right)^{\eps r(n/\log n)} \, ,
\end{align*}
the last inequality since $|B| \ge \log^4 n$ and $\hat{n} < n$. Using that $1-x \le e^{-x}$ this gives
\begin{align*}
\Pr(D[N^-_{\le k}]=H, \tau > \tau'+1)	& \le e^{-\eps r \log^3 n} \sum_F \Pr(D[N^-_{\le k}(v)]=H,D[N^+_{\le \tau'}(u)\setminus V(H)]=F) \\
& \le e^{-\eps r \log^3 n} \Pr(D[N^-_{\le k}(v)]=H).
\end{align*}
The result follows.
\end{proof}
\begin{proof}[Proof of Proposition~\ref{prop:direct}]
Recall that we set $\ell^* =  \log_r n - \log_r \log_r n - 1$, and the notation
$N^*_j=N^+_j(u)\setminus V(H)$, $d^*_j=|N^*_j|$ from Section~\ref{sec:out_tech}.
Once $n$ is large enough that $\ell^*+1 > 5$ we have
\[
\Pr^H(\dist(u,N^-_{\leq k}(v))> \ell^*+1,\overline{\ESL}) =
\Pr^H(\dist(u,N^-_{\leq k}(v))> \ell^*+1,N^+_{\le 5}(u)\cap V(H)=\emptyset,\overline{\ESL}),
\]
and we focus on the latter probability. It is convenient to use the shorthand $E=\{N^+_{\le 5}(u)\cap V(H)=\emptyset\} \cap \overline{\ESL}$.

Now take $\eps \in (0,c)$, where $c$ is the constant from Corollary~\ref{cor:jell}.
and let $\tau'$ be as in Lemma~\ref{lem:tau_tauprime}. Then by that lemma,
\[
\Pr^H(\dist(u,N^-_{\leq k}(v))> \ell^*+1,E)
\le \Pr^H(\tau' > \ell^*, E) + e^{-\epsilon r\log^3 n}\, .
\]
We bound the second probability by
\[
\Pr^H(\tau' > \ell^*, E)
\le  \Pr^H(\tau' > \ell^*, d^*_{j^*} \ge \log_r^3 n) + \Pr^H(d^*_{j^*} <\log_r^3 n, E)\, .
\]
The first term on the right is at most $e^{-C \log^2 n}$ by Corollary~\ref{cor:jell}.
 We further divide the second as
\begin{align*}
 \Pr^H(d^*_{j^*} <\log_r^3 n, E)
 & \le \Pr^H(d^*_{j^*} <\log_r^3 n,d^*_5 \ge 5) + \Pr^H(d^*_5 < 5,E).
\end{align*}
Recall that $\sigma = \inf\{i\ge 5: d_{i}^* < r^{i-5}+5\}$.
If $d^*_{j^*} <\log_r^3 n$ then $d^*_{j^*} < r^{j^*-5}$; if also $d^*_5 \ge 5$ then $5 < \sigma \le j^*$. Lemma~\ref{lem:grow_by_r} then implies that the first probability on the right is $O(n^{-3})$. By the definition of $E$ and by Lemma~\ref{lem:dfive}, the second probability is also $O(n^{-3})$.
Combining all these bounds we obtain $\Pr^H(\dist(u,N^-_{\leq k}(v))> \ell^*+1,\overline{\ESL})=O(n^{-3})$, as required.
\end{proof}
\section{Lower Bound on the Diameter}\label{sec:lower}


In the same spirit as in the previous section, we write $ k^*=k^*(n,\eps)= (\eta_r-\eps/2)\log_r{n}$ and $\ell^*=\ell^*(n,\eps)= (1-\eps/2)\log_r{n}$. In order to derive a lower bound on the diameter of $D$ we will show that for every $\eps>0$ there exist $u,v\in [n]$ such that $\dist(u,v)\geq k^*+\ell^*$. Together with the proof in Section~\ref{sec:upper_proof}, it concludes the proof of Theorem~\ref{thm:main}.

\begin{definition}
For $v \in [n]$, let $k_1=k_1(v)= \min\{k:\; d^-_k(v)\geq \log^4{n}\}$; this is $\infty$ if $d^-_k(v) < \log^4 n$ for all $k$. A vertex $v$ is an \emph{$\eps$-flag} (or simply a flag, if $\eps$ is clear from context) if $k_1 \in [k^*,\infty)$, $|N^-_{\le k_1}(v)| \le \log^7 n$, and $D[N^-_{\le k_1}(v)]$ is a tree. We write $F=F(\eps) \subseteq [n]$ for the set of $\eps$-flags.
\end{definition}
The condition that $D[N^-_{\le k_1}(v)]$ is a tree means that along any shortest path from $N^-_{k_1}(v)$ to $v$, at each node $w$ there are $(r-1)$ possible ``wrong turns'' that lead to $[n]\setminus N^-_{k_1}(v)$. We will use this when bounding $\pi_{\min}$ in Section~\ref{sec:stat}.

The first lemma shows that whp there are no flags outside $D_0=D_0(n,r)$, the attractive strongly connected component of $D(n,r)$.
\begin{lemma}\label{lem:prop_flag}
For every $\eps>0$, $\Pr(F(\eps) \setminus V(D_0) \ne \emptyset) = o(1)$.
\end{lemma}
\begin{proof}%

Fix $\eps > 0$ and write $F=F(\eps)$.
If $D_0$ is attractive then with high probability every vertex $v$ with $\max_{u \in [n]} \dist(u,v) < \infty$ satisfies $v \in V(D_0)$.
Since $D_0$ is attractive whp~\cite{grusho1973limit}, in order to show $\Pr(F\subset V(D_0))=1-o(1)$, it suffices to show that whp, for all $v \in F$ and $u \in [n]$ we have $\dist(u,v) < \infty$.


Let $\mathscr{T}(v)$ be the set of digraphs $T$ with $v \in V(T)$ and $V(T) \subset [n]$ such that if $D[N^-_{\le k_1}(v)]=T$ then $v$ is a flag.
By the definition of a flag, all the elements of $\mathscr{T}(v)$ are rooted at $v$. If $D[N^-_{\le k_1}(v)]=T$ then $N^-_{\le k_1}(v)$ contains no loop vertices.
It follows that $\mathscr{T}(v)$ is precisely the set of graphs $T$ such that $\Pr(D[N^-_{\leq k_1}(v)]=T, v\in F,\overline{\ESL})>0$.

For $T \in \mathscr{T}(v)$ we thus have
\[
\Pr(D[N^-_{\leq k_1}(v)]=T, v\in F, \overline{\ESL}) = \Pr(D[N^-_{\leq k_1}(v)]=T, \overline{\ESL}).
\]
It follows that 
\begin{align*}	
	& \Pr(\exists{u,v} \in [n]: v \in F, \dist(u,v)=\infty)\\
\le	&  \Pr(\ESL) + \sum_{u,v \in [n]} \Pr(v \in F, \dist(u,v)=\infty,\overline{\ESL}) \\
= 	& \Pr(\ESL) + \sum_{u,v \in [n]} \sum_{T \in \mathscr{T}} \Pr(\dist(u,v)=\infty,\overline{\ESL},D[N^-_{\leq k_1}(v)]=T) 
\end{align*}
We bound the inner sum by writing 
\begin{align*}
& \sum_{T \in \mathscr{T}} \Pr(\dist(u,v)=\infty,\overline{\ESL},D[N^-_{\leq k_1}(v)]=T) \\
= & \sum_{T \in \mathscr{T}} \Pr(D[N^-_{\leq k_1}(v)]=T) \cdot \Pr(\dist(u,v)=\infty,\overline{\ESL}\mid D[N^-_{\leq k_1}(v)]=T)  \\
\le & \sup_{T \in \mathscr{T}} \Pr(\dist(u,v)=\infty,\overline{\ESL}\mid D[N^-_{\leq k_1}(v)]=T) \cdot \sum_{T \in \mathscr{T}} \Pr(D[N^-_{\leq k_1}(v)]=T) \\
\le & \sup_{T \in \mathscr{T}} \Pr(\dist(u,v)=\infty,\overline{\ESL}\mid D[N^-_{\leq k_1}(v)]=T)\, ,
\end{align*}
the final bound because a sum of probabilities of disjoint events is at most one. 

Now fix $T \in \mathscr{T}$ and write $h=h(T)$ for the height of $T$ (i.e., the greatest number of edges on a path ending at the root $v$).
Observe that if $D[N^-_{\leq k_1}(v)]=T$ then $k_1=h(T)$.
We thus have the equality of events
\[
\{D[N^-_{\leq k_1}(v)]=T\}\cap \{v\in F\} = \{D[N^-_{\leq k_1}(v)]=T\}=\{D[N^-_{\leq h}(v)]=T\}.
\]
If $D[N^-_{\leq h}(v)]=T$ then the event $E_h=\{d_h^-(v) \ge \log^4 n,d^-_{\le h}(v)\le \log^7 n\}$ from Proposition~\ref{prop:direct} occurs (since in this case $v$ is a flag), so
\[
\{D[N^-_{\leq h}(v)]=T\}=\{D[N^-_{\leq h}(v)]=T\}\cap E_h\, ,
\]
so $\Pr(D[N^-_{\leq h}(v)]=T\},E_h,\overline{\ESL}) > 0$. 
It follows by Proposition~\ref{prop:direct} that $\Pr(\dist(u,v)=\infty,\overline{\ESL}\mid D[N^-_{\leq k_1}(v)]=T) = O(n^{-3})$. 
Using this bound, the result follows from the two preceding inequalities and the fact that $\Pr(\ESL)=\Theta(n^{1-r})=O(n^{-1})$.
%
\end{proof}
We now provide a lower bound for the probability that a fixed vertex is an $\eps$-flag.
\begin{lemma}\label{lem:exit_towers}
For $\eps>0$ sufficiently small, there is $\beta > 0$ such that for $n$ large, $\Pr(v\in F(\eps)) \ge n^{\beta-1}$.
\end{lemma}
\begin{proof}
We assume $n$ large throughout.
Given a tree $T$, let $k_1(T) = \inf\{k: |T_k| \ge \log^4 n\}$, let
$A(T)$ be the event that $|T_{k^*}| \in [\eps \log^3 n,\log^3 n]$, let
$B(T)$ be the event that $\max_{i \le k^*} |T_i| \le \log^6 n$, and
$C(T)$ be the event that $k_1 \le k^* + 5 \log_r \log n$ and $|T_{k_1}| \le \log^5 n$. (We may view a deterministic tree as a random tree in the same way as 
we may view a constant as a random variable, so it is reasonable to call $A(T)$, $B(T)$ and $C(T)$ events even if $T$ is deterministic.) 

We first bound the probability that $A,B$ and $C$ occur for a Poisson$(r)$ Galton-Watson tree $\cT$.
If $\eps$ is sufficiently small then by Lemma~\ref{lem:riordan} there is $a > 0$ such that
\begin{align*}
\Pr(A(\cT)) 		& \ge a\left(r(1-\lambda_r)\right)^{k^*-\log_r \log^3n}\\
			& \ge a r^{\log_r\left(r(1-\lambda_r)\right)(\eta_r-\eps/2-o(1))\log_r n}\\
			& \ge a r^{-(1-\eps/2\eta_r-o(1))\log_r{n}}\\
			& \ge a r^{-(1-\eps/2)\log_r{n}}\\
			& = a n^{\eps/2-1}\, ,
\end{align*}
where we used the value of $k^*$, that $\log_r{(r(1-\lambda_{r}))}=-\eta_r^{-1}$ and that $\eta_r<1$.
Next, if $\overline{B(\cT)}$ occurs then let $i \le k^*$ be minimal such that $|\cT_i| > \log^6 n$.
In order for $A(\cT)$ to additionally occur the number of descendants of $\cT_i$ alive at time $k^*$ must be less than $\log^3 n$.
Writing $p$ for the survival probability of a Poisson$(r)$ branching process, it follows as in the proof of Proposition~\ref{lem:towers} that
\[
\Pr(A(\cT), \overline{B(\cT)}) \le \Pr(\Bin(\log^6 n,p) \le \log^3 n) \le n^{-3}\, ,
\]
the last inequality by a Chernoff bound.

To bound the probability of $\overline{C(\cT)}$, let $N=N(\cT)$ be the number of vertices in $\cT_{k^*}$ with at least one descendant in $\cT_{k^*+5 \log_r\log n}$; if $\cT_{k^*}=\emptyset$ then $N=0$. 
If $C(\cT)$ does not occur then one of the following {\em must} occur. 
\begin{itemize}
\item[(a)] $N < \log^2 n$.
\item[(b)] $N \ge \log^2 n$ but $k_1 > k^* + 5 \log_r\log n$.
\item[(c)] $|\cT_{k_1}| > \log^5 n$. 
\end{itemize}
If $A(\cT)$ occurs then $|\cT_{k^*}| \ge \eps \log^3 n$, so by the branching property (i.e. the independence of subtrees rooted at elements of $\cT_{k^*}$), we have 
\[
\Pr(A(\cT),N < \log^2 n) \le \Pr(\Bin(\eps \log^3 n,p) \le \log^2 n) < n^{-3}\, 
\]
for large $n$, by a Chernoff bound. Next, to have $k_1 > k^*+5\log_r\log n$, every vertex in $\cT_k^*$ must have fewer than $\log^4 n$ descendants in $\cT_{k^*+5\log_r\log n}$, so by Lemma~\ref{lem:riordan} and the branching property we have 
\begin{align*}
\Pr(N \ge \log^2 n, k_1 > k^*+5\log_r\log n) & 
\le (\Pr(|\cT_{5\log_r\log n}| \in (0,\log^4 n)))^{\log^2 n} \\
\le & \left(C(r(1-\lambda_r))^{5\log_r\log n - \log_r \log^4 n}\right)^{\log^2 n} \\
\le & \left(C(r(1-\lambda_r))^{\log\log n}\right)^{\log^2 n} \\
\le & n^{-3} 
\end{align*}
for large $n$, the last inequality because $r(1-\lambda_r) < 1$. 
Finally, by the Markov property and the definition of $k_1$, writing $\Po(t)$ for a Poisson$(t)$ random variable, we have 
\begin{align*}
\Pr(|\cT_{k_1}| > \log^5 n) & \le \sup_{m < \log^4 n} \Pr(\Po(rm)> \log^5 n\mid \Po(rm) \ge \log^4 n) \\
					& \le \Pr(\Po(r\log^4 n) > \log^5 n - \log^4 n)\, .
\end{align*}
 Standard estimates for the Poisson upper tail (see, e.g., \cite{penrose}, Lemma~1.2) then yield 
$\Pr(|\cT_{k_1}| > \log^5 n) \le n^{-3}$. Combining these bounds, we obtain that, for $n$ large, 
\[
\Pr(A(\cT),B(\cT),\overline{C(\cT)}) \le \Pr(\overline{C(\cT)}\mid A(\cT),B(\cT)) \le 3n^{-3}\, .
\]
Combining inequalities, and choosing $\beta > 0$ appropriately, yields
\[
\Pr(A(\cT),B(\cT),C(\cT)) \ge a n^{\eps/2-1} - 4n^{-3} \ge 3n^{\beta-1}.
\]
Now, if $A(\cT),B(\cT)$ and $C(\cT)$ all occur then $|\cT_{\le k_1}| \le k^* \log^6 n <\log^7 n$, so we may use Lemma~\ref{lem:tree_like} to transfer our bound from the Poisson$(r)$ Galton-Watson tree to the tree $T^-_{\le k_1}(v)$. We obtain
\begin{align*}
\Pr(A(T^-_{\le k_1}(v)),B(T^-_{\le k_1}(v)),C(T^-_{\le k_1}(v))) & = \left(1+ O\Big(\frac{\log^{14} n}{n}\Big)\right) \Pr(A(\cT),B(\cT),C(\cT)) \\
 & \ge 2n^{\beta-1}\, .
\end{align*}
Given that $A(T^-_{\le k_1}(v)),B(T^-_{\le k_1}(v))$ and $C(T^-_{\le k_1}(v)$ all occur, in order to have $v \in F(\eps)$ it is sufficient that
$D[N^-_{\le k_1}(v)]$ is a tree, i.e. that $D[N^-_{\le k_1}(v)]=T^-_{\le k_1}(v)$.

Finally, when $A(T^-_{\le k_1}(v)),B(T^-_{\le k_1}(v))$ and $C(T^-_{\le k_1}(v))$ all occur we have $|V(T^-_{\le k_1}(v))| < \log^7 n$. By
Corollary~\ref{cor:ibfs_stoc}, in this case for each element $u \in V(T^-_{\le k_1}(v))$, the probability that there is a non-tree edge $u$ to $V(T^-_{\le k_1}(v))$ is at most $(r \log^7 n)/n$. It follows that
\[
\Pr(A(T^-_{\le k_1}(v)),B(T^-_{\le k_1}(v)),C(T^-_{\le k_1}(v)),D[N^-_{\le k_1}(v)]\ne T^-_{\le k_1}(v)) \le \frac{r \log^{14} n}{n} < n^{\beta-1}.
\]
The result follows.
\end{proof}
%
%
The next lemma is our key tool for controlling joint probabilities of in-neighbourhoods of distinct vertices.
\begin{lemma}\label{lem:approx_ind}
Fix $u,v \in [n]$ and trees $T,T'$, with roots $u$ and $v$, respectively, and with $V(T) \cup V(T') \subset [n]$ and $V(T) \cap V(T')=\emptyset$.
Then
\begin{align*}
\Pr(D[N_{\le h}^-(u)]=T,D[N_{\le h'}^-(v)]=T') =& \left(1+O\left(\frac{|V(T)|^2}{n-|V(T')|}+\frac{|V(T')|^2}{n-|V(T)|}\right)\right)\\
&\cdot\Pr(D[N_{\le h}^-(u)]=T)\cdot \Pr(D[N_{\le h'}^-(v)]=T')\, .
\end{align*}
\end{lemma}
\begin{proof}
Recall that $T_i$ is the $i$-th generation of tree $T$.
Write $h$ and $h'$ for the respective heights of $T$ and $T'$, and $t$ and $t'$ for their respective sizes.
In order that $D[N_{\le h}^-(u)]=T$, it is necessary and sufficient that the following events occur.
\begin{itemize}
\item For each $x \in V(T)\setminus \{u\}$, there is an edge from $x$ to $p_T(x)$ in $D$; call this event $A_1(u,T)$.
\item There are no other edges within $D[V(T)]$; call this event $A_2(u,T)$.
\item There are no edges from $[n]\setminus D[V(T)]$ to $V(T)\setminus T_{h}$; call this event $A_3(u,T)$.
\end{itemize}
Note that $A_3$ is independent of $A_1$ and $A_2$, so we have
\begin{equation}\label{eq:one_guy}
\Pr(D[N_{\le h}^-(u)]=T) = \Pr(A_1(u,T),A_2(u,T)) \cdot \left(\frac{n-t+|T_h|}{n}\right)^{r(n-t)}\, .
\end{equation}
We now consider two such events simultaneously. Observe that if $T'$ has root $v$ and height $h'$, and $V(T') \cap V(T)=\emptyset$, then
$A_1(u,T)\cap A_2(u,T)$ is independent of $A_1(v,T')\cap A_2(v,T')$. We thus have
\begin{align}\label{eq:two_guys}
	& \Pr(D[N_{\le h}^-(u)]=T,D[N_{\le h'}^-(v)]=T') \nonumber\\
= &\Pr(A_1(u,T), A_2(u,T))\cdot \Pr(A_1(v,T'), A_2(v,T'))\nonumber\\
&  \cdot \Pr(A_3(u,T),A_3(v,T')\mid A_1(u,T), A_2(u,T),A_1(v,T'), A_2(v,T'))\, .
\end{align}
Given that $A_1(u,T)$ and $A_2(u,T)$ occur, there are precisely $1+(r-1)t$ edges leaving $V(T)$, and the heads of these edges are uniformly distributed over $[n]\setminus V(T)$. The conditional probability no such edges have head in $V(T')\setminus T'_{h'}$ is
\[
\left(\frac{n-t-t'+|T'_{h'}|}{n-t} \right)^{1+(r-1)t}.
\]
Similar considerations for edges leaving $V(T')$ and edges with tail in $[n]\setminus (V(T)\cup V(T'))$ yield the identity
\begin{align*}
& \Pr(A_3(u,T),A_3(v,T')\mid A_1(u,T), A_2(u,T),A_1(v,T'), A_2(v,T')) \\
= & \left(1-\frac{t'-|T'_{h'}|}{n-t} \right)^{1+(r-1)t}\cdot \left(1-\frac{t-|T_{h}|}{n-t'} \right)^{1+(r-1)t'} \cdot \left(1-\frac{t+t'-|T_{h}|-|T'_{h'}|}{n}\right)^{r(n-t-t')}\, .
\end{align*}
Combined with~\eqref{eq:one_guy} and~\eqref{eq:two_guys}, straightforward arguments give
\begin{align*}
	& \Pr(D[N_{\le h}^-(u)]=T,D[N_{\le h'}^-(v)]=T') \\
= 	& \left(1+O\left(\frac{t^2}{n-t'}+\frac{(t')^2}{n-t}\right)\right)\Pr(D[N_{\le h}^-(u)]=T)\cdot \Pr(D[N_{\le h'}^-(v)]=T')\, .\qedhere
\end{align*}
\end{proof}
\begin{corollary}\label{cor:prod_form}
For distinct $u,v \in [n]$ we have $\Pr(u,v \in F) \le (1+o(1))(\Pr(u \in F) + \log^{15}n/n)\Pr(v \in F)$.
\end{corollary}
\begin{proof}
We first divide according to whether or not $N^-_{\le k_1(u)}(u) \cap N^-_{\le k_1(v)}(v)$ is empty:
\begin{align*}
& \Pr(u,v \in F) \\
 = &  \Pr(u,v \in F,N^-_{\le k_1(u)}(u) \cap N^-_{\le k_1(v)}(v)\ne \emptyset) +
\Pr(u,v \in F,N^-_{\le k_1(u)}(u) \cap N^-_{\le k_1(v)}(v)=\emptyset).
\end{align*}
We start with the first term on the right. If $v \in F$ then $| N^-_{\le k_1(v)}(v)) | \le \log^7 n$, so by symmetry
\[
\Pr(v \in F, u \in N^-_{\le k_1(v)}(v)) \le \frac{\log^7 n}{n}\cdot \Pr(v \in F).
\]
Next, by conditioning on $N^-_{\le k_1(v)}(v)$ we have
\begin{align*}
	& \Pr(u,v \in F,u \not \in N^-_{\le k_1(v)}(v), N^-_{\le k_1(u)}(u) \cap N^-_{\le k_1(v)}(v)\ne \emptyset) \\
\le	& \sum_{\{T \in \mathcal{T}(v): \atop u \not \in V(T)\}} \Pr(D[N^-_{\le k_1(v)}(v)]=T) \cdot \Pr(u \in F,  N^-_{\le k_1(u)}(u) \cap V(T) \ne \emptyset \mid  D[N^-_{\le k_1(v)}(v)]=T)\, .
\end{align*}
In order to bound the final probability, first fix $T$ as in the supremum and suppose that $D[N^-_{\le k_1(v)}(v)]=T$.
Consider the iBFS procedure starting from $u$. Recall that at step $i$, $R^-_i$ is the set of explored vertices and $S^-_i$ is the set of discovered vertices.
Let $i_0= \min\{i: |R^-_i \cup S^-_i| > \log^7 n\}$. If $u \in F$ then $|N^-_{\le k_1(u)}(u)| \le \log^7 n$, so to have $N^-_{\le k_1(u)}(u) \cap V(T)=\emptyset$
it suffices that $(R^-_{i_0-1} \cup S^-_{i_0-1})\cap V(T)=\emptyset$. Since $|V(T)| \le \log^7 n$, by Lemma~\ref{lem:bfs_condition} we thus have
\[
 \Pr(u \in F,  N^-_{\le k_1(u)}(u) \cap V(T) \ne \emptyset \mid  D[N^-_{\le k_1(v)}(v)]=T) \le \frac{r (\log^7 n + 1)\log^{7} n}{n-2\log^7 n}\, .
\]
Together with the two preceding displayed equations, for $n$ large this gives
\begin{align}\label{eq:disjoint_uv}
	& \Pr(u,v \in F,N^-_{\le k_1(u)}(u) \cap N^-_{\le k_1(v)}(v)\ne \emptyset)\nonumber\\
\le 	& \frac{\log^7 n}{n}\cdot \Pr(v \in F) +
\sum_{\{T \in \mathcal{T}(v): u \not \in V(T)\}} \Pr(D[N^-_{\le k_1(v)}(v)]=T) \cdot \frac{r(\log^7 n + 1)\log^{7} n}{n-2\log^7 n}\nonumber\\
\le & \frac{\log^{15} n}{n} \cdot \Pr(v \in F),
\end{align}
the last inequality since 
$$
\sum_{\{T \in \mathcal{T}(v): u \not \in V(T)\}} \Pr(D[N^-_{\le k_1(v)}(v)]=T) = \Pr(v \in F, u \not \in N^-_{\le k_1(v)}(v)])\;.
$$

We now turn to the case that $N^-_{\le k_1(u)}(u)$ and $N^-_{\le k_1(v)}(v)$ are disjoint. We have
\begin{align*}
	& \Pr(u,v \in F,N^-_{\le k_1(u)}(u) \cap N^-_{\le k_1(v)}(v)=\emptyset) \\
= & \mathop{\sum_{\{(T,T') \in \mathcal{T}(u)\times\mathcal{T}(v):}}_{V(T) \cap V(T')=\emptyset\}} \Pr(D[N^-_{\le k_1(u)}(u)]=T,D[N^-_{\le k_1(v)}(v)]=T')\, .\\
= & \left(1+O\pran{\frac{\log^{14} n}{n}}\right)\mathop{\sum_{\{(T,T') \in \mathcal{T}(u)\times\mathcal{T}(v):}}_{V(T) \cap V(T')=\emptyset\}}
\Pr(D[N^-_{\le k_1(u)}(u)]=T) \cdot \Pr(D[N^-_{\le k_1(v)}(v)]=T')\, ,
\end{align*}
the last line by Lemma~\ref{lem:approx_ind}. (Although $k_1(u)$ and $k_1(v)$ are random, by the same argument as in Lemma~\ref{lem:prop_flag} we may replace them by the deterministic values $h(T)$ and $h(T')$  without affecting the probability, so Lemma~\ref{lem:approx_ind} indeed applies.)
Summing over all pairs $(T,T') \in \mathcal{T}(u)\times\mathcal{T}(v)$ gives an upper bound, so we obtain
\[
\Pr(u,v \in F,N^-_{\le k_1(u)}(u) \cap N^-_{\le k_1(v)}(v)=\emptyset) \le (1+o(1)) \Pr(u \in F)\Pr(v \in F)\, .
\]
Together with~\eqref{eq:disjoint_uv} this completes the proof.
\end{proof}
\begin{corollary}\label{lem:many_towers}
For all $\eps > 0$, $\Pr(F(\eps)=\emptyset) =o(1)$.
\end{corollary}
\begin{proof}
By Lemma~\ref{lem:exit_towers} and linearity of expectation there is $\beta > 0$ such that $\E(|F|) =n \Pr(1 \in F) \ge n^{\beta}$.
Next, by Corollary~\ref{cor:prod_form}, for $n$ large we have
\begin{align*}
\E(|F|^2)& = \sum_{u,v\in [n]} \Pr(u,v\in F) \\
		& = n(n-1)\Pr(1,2\in F) +n\Pr(1\in F) \\
		& \le (1+o(1))n(n-1) \Pr(1 \in F)(\Pr(2 \in F) + \log^{15} n/n) + n \Pr(1 \in F) \\
		& \le (1+o(1) (n \Pr(1 \in F))^2 + (n\log^{15} n) \Pr(1 \in F) \\
		& = (1+o(1))  (n \Pr(1 \in F))^2\, .
\end{align*}
The result follows by Chebyshev's inequality.
\end{proof}
\begin{proof}[Proof of the lower bound in Theorem~\ref{thm:main}]
Fix $\eps > 0$ and write $k^* =  (\eta_r - \eps/2)\log_r n$.
Suppose that $D_0$ is attractive, that $|D_0| \ge n/2$, and that $F(\eps) \subset D_0$.
Suppose also that for all $w \in [n]$ and $j \ge 0$, $d^-_{\le j}(w) \le (r+\eps)^j \log_r^2 n$.
Under these assumptions, if $v \in F(\eps)$ then $v \in D_0$. Furthermore, $d^-_{\le k^*}(v) \le \log^7 n$ so for all $j \ge 0$,
\[
d^-_{\le k^*+j}(v) \le (r+\eps)^j \log_r^9 n.
\]
Writing $j_0 = \inf\{j: V(D_0) \subset N^-_{\le k^*+j}(v)\}$, it follows that $(r+\eps)^{j_0} \log_r^9 n \ge n/2$. Provided $\eps$ is chosen small enough, for $n$ large
this implies that $j_0 \ge (1-3\eps/2) \log_r n$, so there is some node $u\in V(D_0)$ with $\dist(u,v) \ge k^*+(1-3\eps/2)\log_r n = (1+\eta_r -2\eps)\log_r n$.
Altogether, this yields 
\begin{align*}
& \Pr(\diam(D_0) < (1+\eta_r -2\eps)\log_r n) \\
\le &
\Pr(D_0\mbox{ is not attractive}) + \Pr(|V(D_0)| < n/2) \\
&+
\Pr(F(\eps)=\emptyset) + \Pr(\exists v \in F(\eps)\setminus D_0) +
\Pr(\exists w \in [n],j \ge 0: d^-_j(w) > (r+\eps)^j \log_r^2 n)
\end{align*}
The first two probabilities were shown to tend to $0$ in \cite{grusho1973limit}. The third tends to $0$ by Corollary~\ref{lem:many_towers}, the fourth by Lemma~\ref{lem:prop_flag},
and the last by Proposition~\ref{prop:UB_in_neighs}. As $\eps>0$ was arbitrarily small, the lower bound on $\diam(D_0)$ follows; since $\diam(D) \ge \diam(D_0)$ so does the lower bound on $\diam(D)$.
\end{proof}

\section{The Stationary Distribution}\label{sec:stat}
In this section we prove Theorem~\ref{thm:dist}.
Recall that $D_0=D_0(n,r)$ is the largest strongly connected component of $D$ and that with high probability $D_0$ is attractive~\cite{grusho1973limit} and ergodic~\cite{ballearxiv}. Write $\pi_{\max}=\pi_{\max}(D_0)$ and $\pi_{\min}=\pi_{\min}(D_0)$.
Also, write $\mathrm{X}=(X_k,k \ge 0)$ for simple random walk on $D=D(n,r)$.

It is important to distinguish the randomness of the graph $D$ from that of the walk $\mathrm{X}$. For $v \in V(D)=[n]$, write $\mathbf{P}_v$ for the (random) probability measure under which $\mathrm{X}$ has the law of simple random walk on $D$ with $X_0=v$, and $\esub{v}$ for the corresponding expectation operator. It is handy to have a concrete description of $\mathrm{X}$ under $\mathbf{P}_v$, as follows. Recall that $D$ has edges $\{(i,L_{i,j}),(i,j) \in [n]\times[r]\}$ (this is the ``canonical construction'' from the introduction). Let $(U_k,k \ge 0)$ be independent and uniformly distributed over $\{1,\ldots,r\}$. Then set
$X_0=v$ and for $k \ge 0$ let $X_{k+1} = L_{X_k,U_k}$.


\subsection{Bounding \texorpdfstring{$\pi_{\max}$}{pimax}}

Fix $k \ge 1$ and view $D[N_{\leq k}^-(v)]$ as a {\em maze}, which a random walk attempting to reach $v$ must navigate. The maze entrances are the elements of $N_k^-(v)$, and the treasure lies at $v$. Suppose that the random walk follows an edge $e$ from $N_{\leq k}^-(v)$ to its complement. After following the edge, the random walk's position has distance greater than $k$ from $v$. Since the distance to $v$ decreases by at most one in a single random walk step, this means that in order to reach $v$ after leaving $N_{\leq k}^-(v)$, the random walk must pass through $N_k^-(v)$: it must restart at one of the maze entrances.

With the preceding paragraph in mind, for positive integer $h$ we say that $D[N_{\leq k}^-(v)]$ is {\em $h$-hard} if for every directed path $P$ from $N_k^-(v)$ to $v$ within $D[N_{\leq k}^-(v)]$, we have
\[
\#\{u \in V(P), |E(u,N_{\le k}^-(v))|=1 \} \ge h.
\]
Perhaps more picturesque: the maze is $h$-hard if no matter what entrance is chosen, along any potential path to the treasure there are at least $h$ locations where only a single direction stays within the maze; the other $(r-1)$ possibilities deposit the searcher outside of the maze walls.

For $S \subset [n]$ let $\tau_S = \inf\{k\ge 0:X_k \in S\}$ and let
$\tau_S^+ = \inf\{k>0:X_k \in S\}$.
\begin{lemma}\label{lem:bound_dist}
For $k \ge 1$, if $D[N_{\leq k}^-(v)]$ is $h$-hard then
\[
\pi(v) \le \frac{1}{r^h \cdot \psub{v}{\tau_{[n]\setminus N_{\leq k}^-(v)} \le \tau_v^+}}.
\]
\end{lemma}
\begin{proof}
 If the maze is $h$-hard then from any $u \in N_k^-(v)$,
\begin{equation}\label{eq:leave_m}
\psub{u}{\tau_v < \tau_{[n]\setminus N_{\leq k}^-(v)}} \le r^{-h}.
\end{equation}
To see this, simply note that in order to have $\tau_v < \tau_{[n]\setminus N_{\leq k}^-(v)}$ the walk must visit at least $h$ vertices
 $w \in N_{\leq k}^-(v)$ with $|E(w,N_{\le k}^-(v))|=1$. But for such a vertex $w$ we have $\psub{w}{X_1 \in N_{\leq k}^-(v)} = 1/r$
 and the inequality follows by the Markov property.

We now use that
\[
\frac{1}{\pi(v)} = \Esub{v}{\tau_v^+} \ge \Esub{v}{\tau_v^+~|~ \tau_{[n]\setminus N_{\leq k}^-(v)} \le \tau_v^+} \cdot \psub{v}{\tau_{[n]\setminus N_{\leq k}^-(v)} \le \tau_v^+}\, .
\]
Let $K$ be the number of visits to $N_k^-(v)$ before the walk visits $v$. Since the inequality~\eqref{eq:leave_m}
holds for all $u \in N_k^-(v)$, it follows that for all $w \in [n]\setminus N_{\leq k}^-(v)$ we have $\Esub{w}{\tau_v} \ge \Esub{w}{K} \ge r^h$. Therefore
\[
\Esub{v}{\tau_v^+~|~ \tau_{[n]\setminus N_{\leq k}^-(v)} \le \tau_v^+} \ge \inf_{w \in [n]\setminus N_{\leq k}^-(v)} \Esub{w}{\tau_v} \ge r^h\, ,
\]
and the result follows.
\end{proof}

\begin{lemma}\label{lem:Msmall}
Fix $\delta>0$ and let $\ell^*=(1-\delta)\log_r{n}$.
Then
\[
\Pr\left(|N_{\le \ell^*}^-(v)|\geq n^{1-\delta/2}\right)= O(n^{-4})\;.
\]
\end{lemma}
\begin{proof}
Choose $\alpha>0$ small, and let $A$ be the event that for all $k\geq 0$ and all $v\in [n]$ we have $d^-_k(v)\leq (r+\alpha)^{k}\log^2{n}$.
By Proposition~\ref{prop:UB_in_neighs}, we have $\Pr(\overline{A})=O(n^{-4})$.
Assuming $\alpha$ is small enough with respect to $\delta$, we also have
\[
|N_{\le \ell^*}^-(v)|=d^-_{\leq \ell^*}(v)=\sum_{k=0}^{\ell^*} d^-_{k}(v) \leq (r+\alpha)^{\ell^*+1}\log^2{n}< n^{1-\delta/2}\, ,
\]
so
\[
\Pr(|N_{\le \ell^*}^-(v)| \ge n^{1-\delta/2})\leq \Pr(|N_{\le \ell^*}^-(v)| \ge n^{1-\delta/2}\mid A)+\Pr(\overline{A})  =  O(n^{-4})\;.\qedhere
\]
\end{proof}

\begin{proposition}\label{prop:hard}
Fix $\delta>0$ and let  $\ell^*=(1-\delta)\log_r n$ and $h=(1-2\delta)\log_r n$. Then
\[
\Pr\left(\bigcap_{v \in [n]} D[N_{\le \ell^*}^-(v)]~\mbox{ is $h$-hard}\right) = 1-O(n^{-3})\, .
\]
\end{proposition}
\begin{proof}
Recall from the introduction that $D=D(n,r)$ has edges $\{(i,L_{i,j}): (i,j) \in [n] \times [r]\}$.
Fix $v \in [n]$.
For each $k \ge 1$ let $T_{\leq k}^-=T_{\leq k}^-(D,v)$ be the iBFS tree of $D[N_{\le k}^-(v)]$ described in Section~\ref{sec:preliminaries}.
For $w \in [n]$, $w \ne v$ let $Y(w) = |E(w,N_{\le \ell^*}^-(v))|-1$, and let $Y(v) = |E(v,N_{\le \ell^*}^-(v))|$.
Observe that there may be multiple edges from $w$ to a vertex $u \in N^+(w)$, so $Y(w)$ may not equal $|N^+(w) \cap N_{\le \ell^*}^-(v)|-1$.

For $w \ne v$, the parent of $w$ in $T^-_{\le \ell^*}$ lies in $N^+(w) \cap N_{\le \ell^*-1}^-(v)\subset N_{\le \ell^*}^-(v)$, so $Y(w) \ge 0$ for all $w \in N_{\le \ell^*}^-(v)$.
The key insight of the proof is that if $D[N_{\le \ell^*}^-(v)]$ is not $h$-hard then there is some simple path $P$ in $D[N_{\le \ell^*}^-(v)]$ from $N_{\ell^*}^-(v)$ to $v$ along which at least $|P|-h$ vertices $w$ have $Y(w) > 0$. To show that $\Pr(D[N_{\le \ell^*}^-(v)]\mbox{ is not $h$-hard})$ is small it thus suffices to show that with high probability no such path exists.

The sets $\{Y(w) : w \in N_{\le \ell^*}^-(v)\}$ are conditionally independent given $T^-{\le \ell^*}$. Furthermore,
given $T^-_{\le \ell^*}$, by Corollary~\ref{cor:ibfs_stoc} we also have
$Y(w) \preceq \Bin(r-1,|N_{\le \ell^*}^-(v)|/n)$ for all  $w \in N_{\le \ell^*}^-(v)$.

Let $A$ be the event that $|N_{\le \ell^*}^-(v)|\leq n^{1-\delta/2}$.
For $S \subset [n]$, it follows that on $A$ the random variable $B(S)=|\{w \in S: Y(w) > 0\}|$ is stochastically dominated by
$\Bin(|S|,(r-1)n^{-\delta/2})$. Furthermore, by Lemma~\ref{lem:Msmall} we have $\Pr(A)=1-O(n^{-4})$.

We would like to conclude as follows. Let $S$ be any path from $T^-_{\ell^*}$ (the last generation of $T^-_{\le \ell^*}$) to $v$.
The arguments of the preceding paragraphs suggest the bound
\[
\Pr(B(S) \ge |S|-h~|~A) \le \Pr(\Bin(|S|,rn^{-\delta/2}) \ge |S|-h) \le 2^{|S|} (rn^{-\delta/2})^{|S|-h}.
\]
On $A$ we have $|T^-_{\ell^*}| < |T^-_{\le \ell^*}| \le n^{1-\delta/2}$, so there are less than $m\cdot r^t$ paths of length $t$ from $T^-_{\ell^*}$ to $v$. Now use the preceding inequality and a union bound over paths of length $t$ and over $t \ge \ell^*$.

To make the preceding argument rigorous, we need to deal with the fact that the set of paths from $T^-_{\ell^*}$ to $v$ are random (even conditional on $T^-_{\ell^*}$, as such paths may follow edges of $D[N_{\le \ell^*}^-(v)]$ which are not edges of $T^-_{\ell^*}$).
To do so, condition on $T^-_{\le \ell^*}$, fix $w \in T^-_{\ell^*}=N_{\ell^*}^-(v)$  and a string $s=s_1s_2\dots s_{t} \in [r]^t$ of length $|s|=t$. This string uniquely specifies a path $P=P(w,s)=(p_i(w,s),0 \le i \le t)$ in $D$: at step $i$ follow the $s_i$-th edge leaving the current vertex.
Formally, we let $p_0=w$ and, for $1 \le i \le t$, let $p_i=L_{p_{i-1},s_i}$.

We reveal the path $P$ edge-by-edge, starting from $w$. By the independence of the sets $Y(u)$, for each $0 \le i < t$, given that the sub-path $p_0,\ldots,p_i$ is simple (in particular $p_i \not \in \{p_0,\ldots,p_{i-1}\}$) then $Y(p_i)$ is conditionally independent of $p_0,\ldots,p_{i-1}$ and
of $Y(p_0),\ldots,Y(p_{i-1})$. It follows that
\[
\Pr(Y(p_i) > 0~|~T_{\le \ell^*},A, (p_0,\ldots,p_i)~\mbox{ simple path in } D[N_{\le \ell^*}^-(v)], (Y(p_j),j<i)) \le rn^{-\delta/2}\, .
\]
By repeated conditioning, we obtain
\begin{align*}
& \Pr(P(w,s)\mbox{ is a simple path in }D[N_{\le \ell^*}^-(v)], B(P(w,s)) \ge t-h~|~T_{\le \ell^*},A) \\
\le&  \Pr(\Bin(t,rn^{-\delta/2}) \ge t-h)\\
\le& 2^{t} (rn^{-\delta/2})^{t-h}\; .
\end{align*}

Now let $A_w(t)$ be the event that there is a simple path $P$ of length $t$ starting from $w$ and staying within $D[N_{\le \ell^*}^-(v)]$, for which $B(P)\geq t-h$. All possible such paths are described by a string $s \in [r]^t$, so by the preceding inequality and a union bound,
\[
\Pr(A_w(t)~|~T_{\le \ell^*},A) \le (2r)^{t} (rn^{-\delta/2})^{t-h} \le \frac{(2r^2)^t}{n^{\delta(t-h)/2}}\, .
\]

Since $\ell^*-h\geq  \delta \log{n}$, this yields that
\[
\Pr\left(\bigcup_{w \in N_{\ell^*}^-(v)} \bigcup_{t \ge \ell^*} A_w(t)~|~T_{\le \ell^*},A\right) \leq n \frac{(2r^2)^{\ell^*}}{n^{\delta(\ell^*-h)/2}} \sum_{j\geq 0} \left(\frac{2r^2}{n^{\delta/2}}\right)^j= O(n^{-4})\, .
\]
Since $\Pr(A) = 1-O(n^{-4})$ this bound also holds unconditionally. But, as described in the first two paragraphs of the proof,
\[
\{D[N_{\le \ell^*}^-(v)]\mbox{ is not $h$-hard}\} \subset \bigcup_{w \in N_{\ell^*}^-(v)} \bigcup_{t \ge {\ell^*}} A_w(t)\, ,
\]
so $\Pr(D[N_{\le \ell^*}^-(v)]\mbox{ is not $h$-hard}) = O(n^{-4})$. A union bound over $v \in [n]$ completes the proof.
\end{proof}
Before proving our bounds on $\pi_{\max}$ we require one final result, which says that with high probability there is at least one escape route along each path from $N^-_{\log\log n}(v)$ to $v$, for all $v$.
\begin{lemma}\label{lem:onehard}
For $v \in [n]$ let $E_v$ be the event that each path from $N^-_{\log\log n}(v)$ to $v$ contains at least one vertex $w$ with $|N^+(w) \cap N^-_{\le \ell^*}(v)|=1$.
Then
\[
\Pr\left(\bigcap_{v \in [n]} E_v \right) = 1-O(n^{-3})\, .
\]
\end{lemma}
Note that the event $E_v$ is not the event that $D[N^-_{\le \log\log n}(v)]$ is $1$-hard: in $E_v$ we require the vertex $w$ to send the searcher
not outside of $D[N^-_{\le \log\log n}(v)]$ but rather out of the larger maze $D[N^-_{\le \ell^*}(v)]$.
The proof of Lemma~\ref{lem:onehard} follows the same lines as that of Proposition~\ref{prop:hard} but is simpler, and is omitted.
\begin{theorem}\label{thm:pimax}
For every $\eps>0$, with high probability we have
\[
\frac{1}{n} \leq \pi_{\max} \leq \frac{1}{n^{1-\eps}}\;,
\]
\end{theorem}
\begin{proof}
The lower bound holds deterministically since $\sum_{v\in [n]} \pi(v)=1$.

To prove the upper bound, fix $v\in [n]$ and  $\delta \in (0,\eps/2)$.
Write $\ell^*=(1-\delta)\log_r{n}$,
and $T_{\leq \ell^*}= T_{\leq \ell^*}(D,v)$.

First, note that if $N^+(v) \setminus N^-_{\le \log\log n}(v) \ne \emptyset$ then $\psub{v}{X_1 \not \in N^-_{\le \log \log n}(v)} \ge 1/r$.
We have
\begin{align*}
& \Pr(N^+(v)\setminus N^-_{\le \log\log n}(v) = \emptyset) \\
\le & \Pr(|N^-_{\le \log \log n}(v)| \ge n^{1/3}) +
\Pr(N^+(v)\setminus N^-_{\le \log\log n}(v) = \emptyset \mid |N^-_{\le \log \log n}(v)| < n^{1/3}) \\
 =& O(n^{-3}) + O(n^{-4/3})\, ,
\end{align*}
the first by Proposition~\ref{prop:UB_in_neighs} and the second by Corollary~\ref{cor:ibfs_stoc}.

Next, if the event $E_v$ from Lemma~\ref{lem:onehard} occurs then for all $w \not \in N^-_{\le \log\log n}(v)$ we have $\psub{w}{\tau_{[n]\setminus N^-_{\le \ell^*}(v)} \le \tau_v} \ge 1/r$.
By the Markov property, it follows that if $N^+(v) \setminus N^-_{\le \log\log n}(v) \ne \emptyset$ and $E_v$, then
\[
\psub{v}{\tau_{[n]\setminus N^-_{\le \ell^*}(v)} \le \tau_v^+}\geq \frac{1}{r} \psub{v}{X_1 \not \in N^-_{\le \log \log n}(v)} \ge \frac{1}{r^2}\, . 
\]
By the preceding paragraph and Lemma~\ref{lem:onehard}, we thus have $\psub{v}{\tau_{[n]\setminus N^-_{\le \ell^*}} \le \tau_v^+} \ge r^{-2}$ with $\Pr$-probability $1-O(n^{-4/3})$.

Finally, if in addition $D[N_{\le \ell^*}^-(v)]$ is $h$-hard then by Lemma~\ref{lem:bound_dist} we obtain that $\pi(v) \le r^{h-2} \le (1+o(1))n^{-(1-2\delta)}$.
By Proposition~\ref{prop:hard} we thus have $\pi(v) \le (1+o(1))n^{-(1-2\delta)}$ with probability $1-O(n^{-4/3})$. A union bound over $v\in [n]$ then completes the proof.
\end{proof}

\subsection{Bounding \texorpdfstring{$\pi_{\min}$}{pimin}}

We bound $\pi_{\min}$ from below using the following lemma.
\begin{lemma}\label{lem:diam}
Let $D$ be any $r$-out regular digraph. If $D$ is ergodic and has diameter $\diam(D) \leq d$, then
\begin{equation*}
\pi_{\min} \geq \frac{1}{1 + d r^d} \enspace.
\end{equation*}
\end{lemma}
\begin{proof}
Fix $v \in V(D)$.
For any $ k\in [d]$ and $u \in N^-_k(D,v)$, let $K(u,k)\geq 1$ be the number of directed paths of length $k$ from $u$ to $v$. Observe that the probability of following each such path is precisely $r^{-k}$, since $D$ is $r$-out regular. Furthermore, since $\pi$ is stationary, it satisfies
\begin{equation*}
\pi(v) \geq \sum_{u \in N^-_k(v)} \pi(u) \cdot \frac{K(u,k)}{r^k}\geq \sum_{u \in N^-_k(v)} \frac{\pi(u)}{r^k} \enspace.
\end{equation*}
By averaging over $ k\in [d]$ we have
\begin{equation*}
\pi(v) \geq \frac{1}{d} \sum_{k=1}^d \sum_{u \in N^-_k (v)} \frac{\pi(u)}{r^k}
\geq \frac{1 - \pi(v)}{d r^d} \enspace ,
\end{equation*}
the last inequality since $\diam(D) \le d$ so $\bigcup_{k=1}^d N_k^-(v) = V(D)\setminus \{v\}$. The lemma follows.
\end{proof}

\begin{theorem}\label{thm:pimin}
For every $\eps>0$ we have
$$
\frac{1}{n^{1+\eta_r+\eps}} \leq \pi_{\min} \leq \frac{1}{n^{1+\eta_r-\eps}}\;,
$$
with high probability.
\end{theorem}

\begin{proof}
Fix $\eps > 0$ small. It is a straightforward consequence of Theorem~\ref{thm:main} and Lemma~\ref{lem:diam} that $\pi_{\min}\geq n^{-(1+\eta_r+\eps)}$ with high probability.
It remains to show that $\pi_{\min}$ is small with high probability.

Let $k^* = (\eta_r - \eps/2) \log_r n$, and recall from Section~\ref{sec:lower} the definition of the set $F=F(\eps)$ of $\eps$-flags. In particular, if $v \in F(\eps)$ then
$D[N^-_{\le k^*}(v)]$ is a tree. It follows that if $v \in F(\eps)$ then $D[N^-_{\le k^*}(v)]$ is $k^*$-hard.

Let $A$ be the event that $\pi_{\max}>n^{-(1-\eps/6)}$. By Corollary~\ref{lem:many_towers}, $\Pr(F=\emptyset)=o(1)$, and by Theorem~\ref{thm:pimax}, $\Pr(\overline{A})=o(1)$. Therefore
\begin{align}\label{eq:final}
\Pr(\pi_{\min}> n^{-(1+\eta_r-\eps)}) 
& = \Pr(\pi_{\min}> n^{-(1+\eta_r-\eps)}, F\ne \emptyset, A)+o(1)\, .
\end{align}
Fix $v \in [n]$,
and for $u \in [n]$ let $K(u)$ be the number of paths of length $k^*$ from $u$ to $v$. Using the stationarity of $\pi$ we have
\[
\pi(v)= \sum_{u\in [n]} \frac{K(u)}{r^{k^*}} \cdot\pi(u)\, .
\]
If $v$ is a flag then $D[N^-_{\le k^*}(v)]$ is a tree,  $K(u)=0$ for $u \not \in N^-_{k^*}(v)$ and $K(u)=1$ for $u \in N^-_{k^*}(v)$.
In this case we also have  $|N^-_{k^*}(v)| \le \log^7 n$. Finally, on $A$ we have $\pi(u) \le n^{-(1-\eps/6)}$.
On the event $\{v \in F\} \cap A$, we thus obtain the bound
\begin{align*}
\pi(v) \leq  |N^-_{k^*}(v)| \cdot \frac{1}{r^{k^*}}\cdot  n^{-(1-\eps/6)}
\le \frac{\log^7{n}}{n^{1+\eta_r-2\eps/3}} \leq \frac{1}{n^{1+\eta_r-\eps}}\;.
\end{align*}
In other words, on $A$, every vertex $v \in F$ {\em deterministically} satisfies $\pi(v) \le n^{-(1+\eta_r-\eps)}$,
so in this case if $F$ is non-empty then $\pi_{\min} \le n^{-(1+\eta_r-\eps)}$.
It follows that the probability on the right of~\eqref{eq:final} is zero, so
$\Pr(\pi_{\min} > n^{-(1+\eta_r-\eps)}) = o(1)$, as required.
\end{proof}

\begin{proof}[Proof of Theorem~\ref{thm:dist}]
The theorem is now an immediate consequence of Theorems~\ref{thm:pimax} and~\ref{thm:pimin}.
\end{proof}

\section*{Acknowledgements}
The first author was supported by an NSERC discovery grant throughout this research. The first author also thanks the Newton Institute for their hospitality during the final stages of the research. The third author wants to thank Xing Shi Cai, Remco van der Hofstad and Bruce Reed for useful discussions. All three authors thank Dana Angluin and Dongqu Chen for useful discussions regarding their forthcoming work \cite{angluincomm}.

\bibliographystyle{plainnat}
\bibliography{paper,randomdfa}


\end{document}